\documentclass[reqno]{amsart}

\usepackage{graphicx}
\usepackage{times}
\usepackage[colorlinks,linkcolor=black,anchorcolor=blue,citecolor=blue]
{hyperref}
\usepackage{centernot}
\usepackage{amsmath}

\newtheorem{theorem}{Theorem}[section]
\newtheorem{lemma}[theorem]{Lemma}
\theoremstyle{definition}
\newtheorem{definition}[theorem]{Definition}

\newtheorem{proposition}[theorem]{Proposition}
\newtheorem{corollary}[theorem]{Corollary}

\theoremstyle{remark}
\newtheorem{remark}[theorem]{Remark}

\numberwithin{equation}{section}

\def\cb{\mathcal{B}}

\def\ce{\mathcal{ E}}

\def\bn{{\mathbb N}}

\def\bq{{\mathbb Q}}
\def\br{{\mathbb R}}

\def\bz{{\mathbb Z}}

\def\m{\mu}

\usepackage{cite}
\usepackage{dsfont}
\usepackage{amssymb}

\begin{document}

\title[Periodic $p$-adic Gibbs measures of $q$-states Potts model on Cayley tree]{Periodic $p$-adic Gibbs measures of $q$-states Potts model on Cayley tree: The chaos implies the vastness of $p$-adic Gibbs measures}

\author{Mohd Ali Khameini Ahmad}
\address{Department of Computational \& Theoretical Sciences, International Islamic University Malaysia, 25200 Kuantan, Pahang, Malaysia}
\email{khameini.ahmad@gmail.com}
\curraddr{LAMA UMR 8050 CNRS, Universit\'e Paris--Est Cr\'eteil, France, 61 Avenue du General de Gaulle, 94010 Creteil Cedex, France}


\author{Lingmin Liao}
\address{LAMA UMR 8050 CNRS, Universit\'e Paris--Est Cr\'eteil, France, 61 Avenue du General de Gaulle, 94010 Creteil Cedex, France}
\email{lingmin.liao@u-pec.fr}

\author{Mansoor Saburov}
\address{Department of Computational \& Theoretical Sciences, International Islamic University Malaysia, 25200 Kuantan, Pahang, Malaysia}
\email{msaburov@gmail.com}

\thanks{The first author (M.A.K.A) is grateful to Embassy of France in Malaysia and Labex B\'ezout for the financial support to pursue his Ph.D at LAMA, Universit\'e Paris--Est Cr\'eteil, France. The third author (M.S.) thanks the Junior Associate Scheme, Abdus Salam International Centre for Theoretical Physics (ICTP), Trieste, Italy, for the invitation and hospitality. He was partially supported by the MOHE grant FRGS14-141-0382.}


\date{}



\begin{abstract}
We study the set of $p$-adic Gibbs measures of the $q$-states Potts model on the Cayley tree of order three. We prove the vastness of the periodic $p$-adic Gibbs measures for such model by showing the chaotic behavior of the correspondence Potts--Bethe mapping over $\mathbb{Q}_p$ for $p\equiv 1 \ (\rm{mod} \ 3)$. In fact, for $0 < |\theta-1|_p < |q|_p^2 < 1$, there exists a subsystem that isometrically conjugate to the full shift on three symbols. Meanwhile, for $0 < |q|_p^2 \leq |\theta-1|_p < |q|_p < 1$, there exists a subsystem that isometrically conjugate to a subshift of finite type on $r$ symbols where $r \geq 4$. However, these subshifts on $r$ symbols are all topologically conjugate to the full shift on three symbols. The $p$-adic Gibbs measures of the same model for the cases $p=2,3$ and the corresponding Potts--Bethe mapping are also discussed.

Furthermore, for $0 < |\theta-1|_p < |q|_p < 1,$ we remark that the Potts--Bethe mapping is not chaotic when $p=2,\ p=3$ and $p\equiv 2 \ (\rm{mod} \ 3)$ and we could not conclude the vastness of the periodic $p$-adic Gibbs measures. In a forthcoming paper with the same title, we will treat the case $0 < |q|_p \leq |\theta-1|_p < 1$ for all $p$.
\end{abstract}

\maketitle

\section{Introduction}

A $p$-adic valued theory of probability (a non--Kolmogorov model in which probabilities take values in the field $\mathbb{Q}_p$ of $p$-adic numbers) was proposed in a series of papers \cite{Kh1990a,Kh1990b,Kh1991,Kh1992a,Kh1992b,Kh1993,Kh1996a,Kh1996b,Kh1999} in order to resolve the problem of the statistical interpretation of $p$-adic valued wave functions in non-Archimedean
quantum physics \cite{BelGas,VVZ,Vol}.  On the other hand, in order to formalize the measure-theoretic approach for $p$-adic probability theory, in the papers \cite{Ilic2015,Ilic2014,Ilic2012,Ilic2016} the authors developed several $p$-adic probability logics which are sound, complete and decidable extensions of the classical propositional logic. This general theory of $p$-adic probability was applicable to the problem of the probability interpretation of quantum theories with non-Archimedean valued wave functions \cite{AKhSh,Kh1994,Kh2009}. The applications of $p$-adic functional and harmonic analysis have also shown up in theoretical physics and quantum mechanics \cite{ACKh,AKhC1997a,AKhC1997b,AKhC1997c,AKhSh,DKhKV}.

Gibbs measure which plays the central role in statistical mechanics is a branch of probability theory that takes its origin from Boltzmann and Gibbs who introduced a statistical approach to thermodynamics to deduce collective macroscopic behaviors from individual microscopic information. Gibbs measure associated
with the Hamiltonian of a physical system (a model) generalizes the notion of a canonical ensemble (see \cite{RUBook}). In the classical case (where the mathematical model was prescribed over the real numbers), the physical phenomenon of phase transition should be reflected in a mathematical model by the non-uniqueness of the Gibbs measures or the size of the set of Gibbs measures for a prescribed model. Due to the convex structure of the set of Gibbs measures over the real numbers field, in order to describe the size of the set of Gibbs measures, it was sufficient to study the number of its extreme elements. Hence, in the classical case, to predict a phase transition, the main attention was paid to finding all possible extreme Gibbs measures. We refer to the book \cite{HOG} for more details.

The rigorous mathematical foundation of the theory of Gibbs measures on Cayley trees was presented in the books \cite{RUBook,RUSurvey}. The $p$-adic counterpart of the theory of Gibbs measures on Cayley trees has also been initiated \cite{GMR,MR1,MR2}. The existence of $p$-adic Gibbs measures as well as the phase transition for some lattice models were established in \cite{M2,MH2013,FMMSOK2015,RUKO2013}. Recently, in \cite{RUKO2015,SMAA2015c,SMAA2015d}, all translation-invariant $p$-adic Gibbs measures of the $p$-adic Potts model on the Cayley tree of order two and three were described by studying allocation of roots of quadratic and cubic equations over some domains of the $\mathbb{Q}_p$.

In the $p$-adic case, due to lack of convex structure of the set of $p$-adic (quasi) Gibbs measures, it is quite difficult to constitute a phase transition with some features of the set of $p$-adic (quasi) Gibbs measures. Moreover, unlike the real case \cite{CKURRK}, the set of $p$-adic Gibbs measures of lattice models on the Cayley tree has a complex structure in a sense that it is strongly tied up with a Diophantine problem (i.e. to find all solutions of a system of polynomial equations or to give a bound for the number of solutions) over the $\mathbb{Q}_p$. In general, the same Diophantine problem may have different solutions from the field of  $p$-adic numbers to the field of real numbers because of the different topological structures.  On the other hand, the rise of the order of the Cayley tree makes difficult to study the corresponding Diophantine problem over the $\mathbb{Q}_p$. In this aspect, the question arises as to whether a root of a polynomial equation belongs to the domains $\mathbb{Z}_p^{*}, \ \mathbb{Z}_p\setminus\mathbb{Z}_p^{*}, \ \mathbb{Z}_p, \ \mathbb{Q}_p\setminus\mathbb{Z}_p^{*}, \ \mathbb{Q}_p\setminus\left(\mathbb{Z}_p\setminus\mathbb{Z}_p^{*}\right), \ \mathbb{Q}_p\setminus\mathbb{Z}_p, \ \mathbb{Q}_p$ or not.  Recently, this problem was fully studied for monomial equations \cite{FMMS}, quadratic equations \cite{SMAA2015c,SMAA2016a}, depressed cubic equations for primes $p>3$ in \cite{FMBOMS,FMBOMSKM,SMAA2016b}. Meanwhile, in \cite{SMAA2013,SMAA2015a,SMAA2015b}, the depressed cubic equations for primes $p=2,3$ are discussed.

This paper can be considered as a continuation of the papers \cite{RUKO2015,SMAA2015c,SMAA2015d}. We are going to study the set of $p$-adic Gibbs measures of the $q$-states Potts model on the Cayley tree of order three. For such study, the conditions $|\theta-1|_p<1$ and $|q|_p\leq 1$ are natural (see Subsections \ref{Sec:2-1} and \ref{Sec:2-5}). We remark that the case $0<|\theta-1|_p<|q|_p=1$ has recently been considered in \cite{FMOK2017} and the regularity of the the dynamics of the Potts--Bethe mapping for primes $p \geq 5,\ p\equiv 2 \ (\rm{mod} \ 3)$ with the condition $0<|\theta-1|_p<|q|_p<1$ has been studied in \cite{SMAA2016Submitted}.  In this paper, we prove the vastness of (periodic) $p$-adic Gibbs measures by showing the chaotic behavior of the Potts--Bethe mapping over $\mathbb{Q}_p$ for primes $p\equiv 1 \ (\rm{mod} \ 3)$ with the condition $0<|\theta-1|_p<|q|_p<1$. To complete our work, in a forthcoming paper with the same title, we will study the Potts--Bethe mapping over $\mathbb{Q}_p$ for the case $0<|q|_p\leq |\theta-1|_p<1$.

We organize our paper as follows. In Section 2, we provide the preliminaries of the paper. In Section 3, we discuss the $p$-adic Gibbs measure of the Potts model associated with an $m$-boundary function. Such a measure is associated to the cycle of length $m$ of the corresponding Potts--Bethe mapping. In Section 4, the dynamical behavior of the Potts--Bethe mapping is studied for $p \equiv 1 \ (\rm{mod} \ 3)$. We find a subsystem that isometrically conjugate to a shift dynamics i.e., the full shift on three symbols or a subshift of finite type on $r\geq 4$ symbols. All these subsystems are proven to be topologically conjugate to the full shift over three symbols. Finally, in Section 5, we discuss the dynamics of Potts--Bethe mapping over $\mathbb{Q}_2$ and $\mathbb{Q}_3$. We prove that except for a fixed point and the inverse images of the singular point, all points converge to an attracting fixed point.
 
\section{Preliminaries}
\subsection{\texorpdfstring{$p$-}{Lg}adic numbers}\label{Sec:2-1}

For a fixed prime $p$, the field $\mathbb{Q}_p$ of $p$-adic
numbers is a completion of the set $\bq$ of rational numbers with
respect to the non-Archimedean norm $|\cdot|_p:\bq\to\br$ given by
\begin{eqnarray*}
|x|_p=\left\{
\begin{array}{c}
p^{-k}, \ x\neq 0,\\
0,\ \quad x=0,
\end{array}
\right.
\end{eqnarray*}
where $x=p^k\frac{m}{n}$ with $k,m\in\bz,$ $n\in\bn$,
$p\centernot\mid m$ and $p\centernot\mid n$. The number $k$ is called \textit{the $p$-order} of $x$
and is denoted by $ord_p(x).$ Note that $ord_p(x)=-\log_p|x|_p=k.$ Moreover, this norm is non-Archimedean because it satisfies the strong triangle inequality: $|x+y|_p\leq \max\left(|x|_p,|y|_p\right).$ The metric on $\mathbb{Q}_p$ induced by this norm, $d(x,y)=|x-y|_p,$ satisfies the ultrametric property: for all $x,y,z \in \mathbb{Q}_p,\ d(x,y) \leq \max\left(d(x,z), d(z,y)\right).$

We respectively denote the set of all {\it $p$-adic integers} and {\it $p$-adic units} of
$\bq_p$ by
$\bz_p=\{x\in\bq_{p}: |x|_p\leq1\}$ and  $\bz_p^{*}=\{x\in\bq_{p}: |x|_p=1\}.$ Any $p$-adic number $x\in\bq_p$ can be uniquely written in the
canonical form
$
x=p^{ord_p(x)}\left(x_0+x_1\cdot p +x_2\cdot p^2+\dots \right)
$
where $x_0\in \{1,2,\dots, p-1\}$ and $x_i \in \{ 0,1,2,\dots, p-1 \}$ for
$i\in \mathbb{N}:=\{1,2,\dots\}$. Therefore, $x =\frac{x^{*}}{|x|_p}$, for some $x^{*}\in\bz_p^{*}$.

We denote by $\mathbb{B}_r(a)=\{x\in \bq_p : |x-a|_p< r\}$ and $\mathbb{S}_{r}(a)=\{x\in \bq_p : |x-a|_p=r\}$ the open ball and the sphere with center $a\in \bq_p$ and radius $r>0$ respectively. Note that, by non-Archimedean property, an open ball is also closed. The $p$-adic logarithm function $\ln_p(\cdot) :\mathbb{B}_1(1)\to \mathbb{B}_1(0)$ is defined by
$$
\ln_p(x)=\ln_p(1+(x-1))=\sum_{n=1}^{\infty}(-1)^{n+1}\frac{(x-1)^n}{n}.
$$ The $p$-adic exponential function $\exp_p(\cdot): \mathbb{B}_{p^{-1/(p-1)}}(0)\to \mathbb{B}_1(1)$ is defined by
$$
\exp_p(x)=\sum_{n=0}^{\infty}\frac{x^n}{n!}.
$$

\begin{lemma}[see \cite{NK,VVZ}]\label{21}  Let $x\in
	\mathbb{B}_{p^{-1/(p-1)}}(0).$ Then we have $$ |\exp_p(x)|_p=1,\ \ 
	|\exp_p(x)-1|_p=|x|_p<1, \ \ |\ln_p(1+x)|_p=|x|_p<p^{-1/(p-1)},$$
	$$ \ln_p(\exp_p(x))=x, \ \ \exp_p(\ln_p(1+x))=1+x. $$
\end{lemma}

Let $\ce_p=\{ x\in\bq_p:   |x-1|_p<p^{-1/(p-1)}\}.$ Obviously, if $p\geq3$ then $\ce_p=\{ x \in \mathbb{Z}_p^{*} : |x-1|_p<1\}=\mathbb{B}_1(1).$ Due to Lemma \ref{21}, we have the following result.

\begin{lemma}[see \cite{NK,VVZ}]\label{epproperty}
	The set $\ce_p$ has the following properties:
	\begin{itemize}
		\item[(i)] $\ce_p$ is a group under multiplication;
		\item[(ii)] one has $|a-b|_p<1$ for all $a,b\in\ce_p$;
		\item[(iii)] if $a,b\in\ce_p$, then one has
		$
		|a+b|_p=\left\{\begin{array}{ll}
		\frac{1}{2}, & \mbox{if }\ p=2\\
		1, & \mbox{if }\ p\neq2;
		\end{array}\right.
		$
		\item[(iv)]  if $a\in\ce_p$, then
		there exists  $h\in \mathbb{B}_{p^{-1/(p-1)}}(0)$ such that
		$a=\exp_p(h)$.
	\end{itemize}
\end{lemma}

Note that a compact-open set in $\mathbb{Q}_p$ is a finite union of balls.
The following is the definition of local scaling on a compact-open set $X \subset \mathbb{Q}_p$.
\begin{definition}[Definition 4.1 of \cite{Kinsbery2009}] \label{def:locally-scaling}
Let $X = \bigcup\limits_{i=1}^{n}\mathbb{B}_{r_i}(a_i) \subset \mathbb{Q}_p$ be a compact-open set. We say that a mapping $f: X\to X$ is locally scaling for $r_i \in \{p^n, n\in \mathbb{Z}\}$ if there exists a function $S:X\to \mathbb{R}_{\geq 1}$ such that for any $x,y \in \mathbb{B}_{r_i}(a_i)$, $S(x)=S(y)=S(a_i)$ and $$|f(x)-f(y)|_p = S(a_i)|x-y|_p.$$
\end{definition}
The function $S$ is called a scaling function. The following lemma will be useful.
\begin{lemma}[Theorem 5.1 of \cite{Kinsbery2009}] \label{scaling}
Let $X= \bigcup\limits_{i=1}^{n}\mathbb{B}_{r_i}(a_i) \subset \mathbb{Q}_p$ be a compact-open and $f: X\to X$ be a scaling for $r_i \in \{p^n, n\in \mathbb{Z}\}.$ Let $S:X\to \mathbb{R}_{\geq 1}$ be the correspondence scaling function. Then 
for all $r^{\prime} \leq r_i$, the restricted map 
$$f : \mathbb{B}_{r^\prime}(a_i) \to \mathbb{B}_{r^{\prime}S(a_i)}(f(a_i))$$ is a bijection.
\end{lemma}

The following result is a direct consequence of Lemma \ref{scaling}.

\begin{corollary} \label{scaling2}
Keep the same assumption as in Lemma \ref{scaling}, we have for all $r^{\prime} \leq r_i$, the restricted map 
$$f : \mathbb{S}_{r^\prime}(a_i) \to \mathbb{S}_{r^{\prime}S(a_i)}(f(a_i))$$ is a bijection.
\end{corollary}

\begin{proof}
Recall $\mathbb{S}_{r^{\prime}}(a_i)=\{x\in \bq_p : |x-a_i|_p=r^{\prime}\}.$ Then $\mathbb{S}_{r^{\prime}}(a_i)=\mathbb{B}_{pr^\prime}(a_i) \setminus \mathbb{B}_{r^\prime}(a_i)$ is a difference of consecutive balls. By Lemma \ref{scaling}, $f$ are bijections from balls to balls. Thus when restricted to the spheres $f$ is also bijective.
\end{proof}

\subsection{\texorpdfstring{$p$-}{Lg}adic subshift}

In this paper, we will apply a theorem in {\cite{Liao1}. We take the same notation as in {\cite{Liao1}.
Let $f : X \to \mathbb{Q}_p$ be a mapping from a compact open set $X \subset \mathbb{Q}_p$ into $\mathbb{Q}_p$. We assume that (i) $f^{-1}\left(X\right) \subset X$ and (ii) $X = \bigcup_{j \in I}\mathbb{B}_{p^{-\tau}}(a_j)$ can be written as a finite disjoint union of balls of centers $a_j$ and of the same radius $p^{-\tau},\ \tau \in \mathbb{Z}$ such that for each $j \in I$ there is an integer $\tau_j \in \mathbb{Z}$ such that for any $x,y \in \mathbb{B}_{r}(a_j)$
\begin{equation}\label{2.1}
\left|f(x)-f(y)\right|_p=p^{\tau_j}|x-y|_p.
\end{equation}
For such a map $f$, we define its Julia set by
\begin{equation}\label{2.2}
J_{f} = \bigcap_{n=0}^{\infty}f^{-n}\left(X\right).
\end{equation}
It is clear that $f^{-1}\left(J_{f}\right) = J_{f}$ and $f\left(J_{f}\right) \subset J_{f}$.

The triple $(X, J_{f}, f)$ is called \textit{a $p$-adic weak repeller} if all $\tau_j$ in \eqref{2.1} are nonnegative, but at least one is positive. It is called \textit{a $p$-adic repeller} if all $\tau_j$ in \eqref{2.1} are positive. For any $i \in I,$ we let $$I_i := \left\{j \in I : \mathbb{B}_{r}(a_j) \cap f\left(\mathbb{B}_{r}(a_i)\right) \neq \emptyset \right\} = \left\{j \in I : \mathbb{B}_{r}(a_j) \subset f\left(\mathbb{B}_{r}(a_i)\right) \right\}$$ (the second equality holds because of the expansiveness and the ultrametric property). We then define a so-called \textit{incidence matrix} $A = (a_{ij})_{I \times I}$ as follows
$$
a_{ij}=
\begin{cases}
1 & \text{if} \ \ j \in I_i, \\
0 & \text{if} \ \ j \not\in I_i.
\end{cases}
$$
If $A$ is irreducible, we say that $(X,J_f,f)$ is \textit{transitive}. Recall that $A$ is irreducible if for any pair $(i,j) \in I \times I$ there is a positive integer $m$ such that $a_{ij}^{(m)} > 0$, where $a_{ij}^{(m)}$ is the entry of the matrix $A^{m}$. For an index set $I$ and an irreducible incidence matrix $A$ given above, we denote by $$\Sigma_A = \left\{(x_k)_{k \geq 0} : x_k \in I,\ A_{x_k,x_{k+1}}=1,\ k \geq 0 \right\}
$$ the subshift space. We equip $\Sigma_A$ with a metric $d_f$ depending on the dynamics which is defined as follows. For $i,j \in I,\ i \neq j$ let $\kappa(i,j)$ be the integer such that $|a_i-a_j|_p=p^{-\kappa(i,j)}$. It is clear that $\kappa(i,j) < \tau$. By the ultrametric inequality, we have $$|x-y|_p = |a_i-a_j|_p,\ i \neq j,\ \forall x \in \mathbb{B}_r(a_i),\ \forall y \in \mathbb{B}_r(a_j).$$
For $x=(x_0,x_1,\cdots,x_n,\cdots) \in \Sigma_A$ and $y=(y_0,y_1,\cdots,y_n,\cdots) \in \Sigma_A$, we define
$$
d_f(x,y)=
\begin{cases}
p^{-\tau_{x_0}-\tau_{x_1}-\cdots-\tau_{x_{n-1}}-\kappa(x_n,y_n)} & \text{if} \ \ n \neq 0, \\
p^{-\kappa(x_0,y_0)} & \text{if} \ \ n = 0
\end{cases}
$$
where $n = n(x,y) = \min\{k \geq 0 : x_k \neq y_k \}$. It is clear that $d_f$ defines the same topology as the classical metric defined by $d(x,y)=p^{-n(x,y)}$.  Let $\sigma$ be the (left) shift transformation on $\Sigma_A$. Then $(\Sigma_A,\sigma,d_f)$ becomes a dynamical system.

\begin{theorem}[{\cite{Liao1}}]\label{conjugacytheorem}
Let $(X,J_f,f)$ be a transitive $p$-adic weak repeller with incidence matrix $A$. Then the dynamics $(J_f,f,|\cdot|_p)$ is isometrically conjugate to the shift dynamics $(\Sigma_A,\sigma,d_f)$.
\end{theorem}

The shift dynamics $(\Sigma_A,\sigma)$ is called a subshift of finite type determined by matrix $A$. When $A$ is a $n\times n$ matrix whose coefficients are all equal to $1$, $(\Sigma_A,\sigma)$ is called a full shift over an alphabet of $n$ symbols, and is usually denoted by $(\Sigma_n,\sigma)$.

\subsection{\texorpdfstring{$p$}{Lg}-adic measure}

Let $(X,\cb)$ be a measurable space, where $\cb$ is an algebra of
subsets $X$. A function $\m:\cb\to \bq_p$ is said to be a {\it
$p$-adic measure} if for any $A_1,\dots,A_n\in\cb$ such that
$A_i\cap A_j=\emptyset$ ($i\neq j$) the following equality holds
$$
\mu\bigg(\bigcup_{j=1}^{n} A_j\bigg)=\sum_{j=1}^{n}\mu(A_j).
$$

A $p$-adic measure is called a {\it probability measure} if
$\mu(X)=1$.  A $p$-adic
probability measure $\m$ is called {\it bounded} if
$\sup\{|\m(A)|_p : A\in \cb\}<\infty $. We are interested in this
important class of $p$-adic measures in which the boundedness condition itself
provides a fruitful integration theory. The reader may refer to \cite{Kh1996b,Kh1999,Kh2009} for more detailed
information about $p$-adic measures.

\subsection{Cayley tree}

Let $\Gamma^k_+ = (V,L)$ be a semi-infinite Cayley tree of order
$k\geq 1$ with the root $x^0$ (each vertex has exactly $k+1$
edges except for the root $x^0$ which has $k$ edges). Let $V$ be
a set of vertices and $L$ be a set of edges. The vertices $x$
and $y$ are called {\it nearest neighbors} if there exists an edge $l\in L$ connecting them. This edge is also denoted by
$l=\langle{x,y}\rangle.$ A collection of
the pairs $\langle{x,x_1}\rangle,\cdots,\langle{x_{d-1},y}\rangle$ is called {\it a path} between vertices $x$ and $y$. The distance $d(x,y)$ between $x,y\in V$ on
the Cayley tree is the length of the shortest path between $x$ and $y$. Let 
$$
W_{n}=\left\{ x\in V: d(x,x^{0})=n\right\}, \quad
V_n=\overset{n}{\underset{m=1}{\bigcup}}W_{m}, \quad L_{n}=\left\{
<x,y>\in L: x,y\in V_{n}\right\}.
$$
The set of direct successors of $x$ is defined by
$$
\forall\  x\in W_{n},\ \ S(x)=\left\{ y\in W_{n+1}:d(x,y)=1\right\}.
$$

We now introduce a coordinate structure in $V$.
Every vertex $x\neq x^{0}$ has the coordinate
$(i_1,\cdots,i_n)$ where $i_m\in\{1,\dots,k\},\ 1\leq m\leq n$ and
the vertex $x^0$ has the coordinate $(\emptyset)$. More precisely, the symbol $(\emptyset)$ constitutes level $0$ and the coordinates
$(i_1,\cdots,i_n)$ form level $n$ of $V$ (from the root $x^0$). In this case, for any $x\in V,\
x=(i_1,\cdots,i_n)$ we have
$$
S(x)=\{(x,i): 1\leq i\leq k\},
$$
where $(x,i)$ means $(i_1,\cdots,i_n,i)$. Let us define a binary operation $\circ:V\times V\to V$
as follows: for any two elements $x=(i_1,\cdots,i_n)$ and $y=(j_1,\cdots,j_m)$ we define
$$
x\circ y=(i_1,\cdots,i_n)\circ(j_1,\cdots,j_m)=(i_1,\cdots,i_n,j_1,\cdots,j_m)
$$
and
$$
y\circ x=(j_1,\cdots,j_m)\circ(i_1,\cdots,i_n)=(j_1,\cdots,j_m,i_1,\cdots,i_n).
$$
Then  $(V,\circ)$ is a noncommutative semigroup with the unit $x^0=(\emptyset)$.
Now, we can define a translation $\tau_g: V\to V$ for $ g\in V$
as follows $\tau_g(x)=g\circ x.$
Consequently, by means of $\tau_g,$ we define a translation $\tilde\tau_g:L\to L$ as $
\tilde\tau_g(\langle{x,y}\rangle)=\langle{\tau_g(x),\tau_g(y)}\rangle.$

Let $G\subset V$ be a sub-semigroup of $V$ and $h:L\to\bq_p$ be a function.
We say that $h$ is $G$-{\it periodic} if $h(\tilde\tau_g(l))=h(l)$ for all $g\in G$ and $l\in L$.
Any $V$-periodic function is called {\it translation-invariant}.

\subsection{\texorpdfstring{$p$-adic $q$}{Lg}-states Potts model}\label{Sec:2-5}

Let $\Phi=\{1,2,\cdots, q\}$ be a finite set.
A configuration (resp. a finite volume configuration, a boundary configuration) is a function $\sigma:V\to \Phi$ (resp. $\sigma_n:V_n\to \Phi,$ $\sigma^{(n)}:W_n\to \Phi$). We denote by $\Omega$ (resp. $\Omega_{V_n}$, $\Omega_{W_n}$) the set of all configurations (resp. all finite volume configurations, all boundary configurations). For given configurations $\sigma_{n-1}\in\Omega_{V_{n-1}}$ and $\sigma^{(n)}\in\Omega_{W_n}$, we define their concatenation to be a finite volume configuration $\sigma_{n-1}\vee\sigma^{(n)}\in\Omega_{V_n}$  such that
$$
\left(\sigma_{n-1}\vee\sigma^{(n)}\right)(v)=
\begin{cases}
\sigma_{n-1}(v) & \text{if} \ \ v\in V_{n-1} \\
\sigma^{(n)}(v) & \text{if} \ \ v\in W_n
\end{cases}.
$$

The Hamiltonian of \textit{$q$-states Potts model} with the spin value set $\Phi=\{1,2,\cdots q\}$ on the finite volume configuration is defined as follows
\begin{eqnarray*}\label{Hamiltonian}
H_n(\sigma_n)=J\sum\limits_{\langle x,y\rangle\in L_n}\delta_{\sigma_n(x)\sigma_n(y)},
\end{eqnarray*}
for all $\sigma_n\in \Omega_{V_n}$ and $n\in\mathbb{N}$ where $J\in \mathbb{B}_{ p^{-1/(p-1)}}(0)$ is a coupling constant, $\langle x,y\rangle$ stands for nearest neighbor vertices, and $\delta$ is Kronecker's symbol. 

\subsection{\texorpdfstring{$p$}{Lg}-adic Gibbs measure}

Let us present a construction of a $p$-adic Gibbs measure of the $p$-adic $q$-states Potts model. We define a $p$-adic measure $\mu^{(n)}_{\mathbf{\tilde{h}}}:\Omega_{V_n}\to\mathbb{Q}_p$ associated with a boundary function $\mathbf{\tilde{h}}:V\to\bq_p^{q},$ $\mathbf{\tilde{h}}(x)=\left({\tilde{h}}^{(1)}_{x},\dots, {\tilde{h}}^{(q)}_{x}\right),$ $x\in V$ as follows
\begin{eqnarray}\label{mu_h^n}
\mu_{\mathbf{\tilde{h}}}^{(n)}(\sigma_n)
=\frac{1}{\mathcal{Z}^{(n)}_{\mathbf{\tilde{h}}}}\exp_p\left\{H_n(\sigma_n)+\sum_{x\in W_n}{\tilde{h}}^{(\sigma_n(x))}_{x}\right\}
\end{eqnarray}
for all $\sigma_n\in \Omega_{V_n}, \ n\in\mathbb{N}$ where $\exp_p(\cdot):\mathbb{B}_{p^{-1/(p-1)}}(0)\to \mathbb{B}_1(1)$ is the $p$-adic exponential function and $\mathcal{Z}^{(n)}_{\mathbf{\tilde{h}}}$ is {\it the partition function} defined by
\begin{eqnarray*}\label{ZN}
	\mathcal{Z}^{(n)}_{\mathbf{\tilde{h}}}=\sum_{\sigma_n\in\Omega_{V_n}}\exp_p\left\{H_n(\sigma_n)+\sum_{x\in W_n}{\tilde{h}}^{(\sigma_n(x))}_{x}\right\}
\end{eqnarray*}
for all $n\in\mathbb{N}.$ Throughout this paper, we consider $\mathbf{\tilde{h}}(x)\in\left(\mathbb{B}_{p^{-1/(p-1)}}(0)\right)^{q}$ for any $x\in V.$

The $p$-adic measures \eqref{mu_h^n} are called \textit{compatible} if one has that
\begin{equation}\label{compatibility}
\sum_{\sigma^{(n)}\in \Omega_{W_n}}\mu_{\mathbf{\tilde{h}}}^{(n)}(\sigma_{n-1}\vee \sigma^{(n)})=\mu_{\mathbf{\tilde{h}}}^{(n-1)}(\sigma_{n-1})
\end{equation}
for all $\sigma_{n-1}\in \Omega_{V_{n-1}}$ and $n\in \mathbb{N}$.

Due to Kolmogorov's extension theorem of the $p$-adic measures \eqref{mu_h^n} (see \cite{GMR,Kh2002,LK}), there exists a unique $p$-adic measure $\mu_{\mathbf{\tilde{h}}}:\Omega\to\mathbb{Q}_p$ such that
$$\mu_{\mathbf{\tilde{h}}}(\{\sigma\mid_{V_n}=\sigma_n\})=\mu_{\mathbf{\tilde{h}}}^{(n)}(\sigma_n), \quad \forall \ \sigma_n\in \Omega_{V_n}, \ \forall\  n\in\mathbb{N}.$$ 
This uniquely extended measure $\mu_{\mathbf{\tilde{h}}}:\Omega\to\mathbb{Q}_p$ is called \textit{a $p$-adic Gibbs measure}. The following theorem describes the condition on the boundary function $\mathbf{{\tilde{h}}}:V\to\bq_p^{q}$ in which the compatibility condition \eqref{compatibility}  is satisfied.

\begin{theorem}[\cite{MR1,MR2}]\label{Equationforh} Let $\mathbf{{\tilde{h}}}:V\to\bq_p^{q},$ $\mathbf{\tilde{h}}(x)=\left(\tilde{h}^{(1)}_{x},\cdots, \tilde{h}^{(q)}_{x}\right)$ be a given boundary function and $\mathbf{{{h}}}:V\to\bq_p^{q-1},$ $\mathbf{h}(x)=\left(h^{(1)}_{x},\cdots, h^{(q-1)}_{x}\right)$ be a function defined as $h^{(i)}_{x}={\tilde{h}}^{(i)}_{x}-{\tilde{h}}^{(q)}_{x}$ for all $ i=\overline{1,q-1}.$ Then the $p$-adic probability distributions $\left\{\mu_{\mathbf{\tilde{h}}}^{(n)}\right\}_{n\in\mathbb{N}}$  are compatible if and only if
	\begin{equation}\label{systemofequationforh}
	\mathbf{{h}}(x)=\sum_{y\in S(x)}\mathbf{F}(\mathbf{h}(y)), \quad \forall \ x\in V\setminus\{x^0\},
	\end{equation}
	where $S(x)$ is the set of direct successors of $x$ and the function $\mathbf{F}:\mathbb{Q}_p^{q-1}\to \mathbb{Q}_p^{q-1},$ $\mathbf{F}(\mathbf{h})=(F_1,\cdots,F_{q-1})$ for $ \mathbf{h}=(h_1, \cdots,h_{q-1})$ is defined as follows
$$F_i=\ln_p\left(\frac{(\theta-1)\exp_p(h_i)+\sum_{j=1}^{q-1}\exp_p(h_j)+1}{\theta+ \sum_{j=1}^{q-1}\exp_p(h_j)}\right), \quad \theta=\exp_p(J).$$
\end{theorem}

\subsection{Translation-invariant \texorpdfstring{$p$}{Lg}-adic Gibbs measures} A $p$-adic Gibbs measure is translation-invariant (in short TIpGM) if and only if the boundary function $\mathbf{{\tilde{h}}}:V\to\bq_p^{q}$ is constant, i.e., $\mathbf{{\tilde{h}}}(x)=\mathbf{\tilde{h}}$ for $x\in V.$ In this case, the condition \eqref{systemofequationforh} takes the form $\mathbf{h}=k\mathbf{F}(\mathbf{h})$ or equivalently
$$
h_i=\ln_p\left(\frac{(\theta-1)\exp_p(h_i)+\sum_{j=1}^{q-1}\exp_p(h_j)+1}{\theta+ \sum_{j=1}^{q-1}\exp_p(h_j)}\right)^k, \quad 1\leq i \leq q-1.
$$ 
Let $\mathbf{z}=(z_1,\cdots,z_{q-1})\in\mathcal{E}_p^{q-1}$ such that $z_i=\exp_p(h_i)$ for any $1 \leq i \leq q-1.$ In what follows, we write $\mathbf{z}=\exp_p(\mathbf{h}).$ Then we have the following result.

\begin{theorem}[\cite{MR1,MR2}]\label{existenceTIpGM}
Let $\mathbf{{\tilde{h}}}:V\to\bq_p^{q}$ be a boundary function, $\mathbf{\tilde{h}}(x)=\mathbf{\tilde{h}}=(\tilde{h}_1,\cdots, \tilde{h}_{q})$. There exists a translation-invariant $p$-adic Gibbs measure $\mu_{\mathbf{\tilde{h}}}:\Omega\to\mathbb{Q}_p$ associated with a boundary function $\mathbf{{\tilde{h}}}$ for any $x\in V$ and $\mathbf{h}=(h_1,\cdots, h_{q-1})$ such that $h_i=\tilde{h}_i-\tilde{h}_q$ for $1\leq i \leq q-1$ if and only if $(z_1,\cdots,z_{q-1})=\mathbf{z}=\exp_p(\mathbf{h})=\left(exp_p(h_1),\cdots, exp_p(h_{q-1})\right)\in\mathcal{E}_p^{q-1}$ is a solution of the following system of equations 
	\begin{equation}\label{equationwrtz}
	z_i=\left(\frac{(\theta-1)z_i+\sum_{j=1}^{q-1}z_j+1}{\theta+ \sum_{j=1}^{q-1}z_j}\right)^k, \quad 1\leq i \leq q-1.
	\end{equation}
\end{theorem}


Let $\mathbf{I}_{q-1}=\{1,\cdots,q-1\}$. For $j\in\mathbf{I}_{q-1}$, let $\mathbf{e}_j=(\delta_{1j},\delta_{2j},\cdots \delta_{q-1j})\in\mathbb{Q}_p^{q-1}$ and for $\alpha \subset \mathbf{I}_{q-1}$, let $\mathbf{e}_\alpha=\sum\limits_{j\in\alpha}\mathbf{e}_j$. For any $\mathbf{z}=(z_1,\cdots,z_{q-1})\in\mathbb{Q}_p^{q-1}$, let $\{\alpha_j(\mathbf{z})\}_{j=1}^d$ be a disjoint partition of the index set  $\mathbf{I}_{q-1}$, i.e., $\bigcup\limits_{j=1}^d\alpha_j(\mathbf{z})=\mathbf{I}_{q-1}, \ \alpha_{j_1}(\mathbf{z})\cap\alpha_{j_2}(\mathbf{z})=\emptyset$ for $j_1\neq j_2$ such that $z_{i_1}=z_{i_2}$ for all $i_1,i_2\in \alpha_j(\mathbf{z})$ and  $z_{i_1}\neq z_{i_2}$ for all $i_1\in \alpha_{j_1}(\mathbf{z})$, $i_2\in \alpha_{j_2}(\mathbf{z})$. Then \begin{eqnarray}\label{foranyk}
\mathbf{z}=\sum\limits_{j=1}^dz_j^{\circ}\mathbf{e}_{\alpha_j(\mathbf{z})}
\end{eqnarray}
where $z_i=z_j^{\circ}$ for any $i\in\alpha_j(\mathbf{z})$ and $1\leq j \leq d.$
 
\begin{theorem}[See \cite{SMAA2015d}]\label{relationdk}
 If $\mathbf{z}\in\mathcal{E}_p^{q-1}$ is a solution of the system  \eqref{equationwrtz} then $1\leq d\leq k.$
\end{theorem}

The descriptions of all TIpGMs for the cases $k=2,3$ were given in \cite{RUKO2015,SMAA2015d}.

\begin{theorem}[Descriptions of TIpGMs for $k=2,$ \cite{RUKO2015}]\label{Descriptionq>2}
	There exists a TIpGM $\mu_{\mathbf{\tilde{h}}}:\Omega\to\mathbb{Q}_p$ associated with a boundary function  $\mathbf{\tilde{h}}=(\tilde{h}_1,\cdots, \tilde{h}_{q})$ if and only if $\tilde{h}_j=\ln_p(hz_j)$ for all $1\leq j \leq q-1$ and $\tilde{h}_q=\ln_p(h)$ where $h\in\mathcal{E}_p$ and $\mathbf{z}=(z_1,\cdots z_{q-1})\in \mathcal{E}_p^{q-1}$ is defined as either one of the following form
	\begin{itemize}
		\item[(\textbf{A})] $\mathbf{z}=\left(1,\cdots,1\right);$
		\item[(\textbf{B})] $\mathbf{z}=\left(z,\cdots,z\right)$ where $z\in\mathcal{E}_p\setminus\{1\}$ is a root of the following quadratic equation 
		\begin{eqnarray*}
		(q-1)^2z^2+\left(2(q-1)-(\theta-1)^2\right)z+1=0;
		\end{eqnarray*}
		\item[(\textbf{C})] $\mathbf{z}=\mathbf{e}_{\alpha_1}+z\mathbf{e}_{\alpha_2}$ with $|\alpha_i|=m_i, \ m_1+m_2=q-1$ such that  $z\in\mathcal{E}_p\setminus\{1\}$ is a root of the following quadratic equation
		\begin{eqnarray*}
		m_2^2z^2+\left[2m_2(m_1+1)-(\theta-1)^2\right]z+(m_1+1)^2=0.
		\end{eqnarray*}
		\end{itemize}
\end{theorem}

\begin{theorem}[Descriptions of TIpGMs for $k=3,$ \cite{SMAA2015d}]\label{Descriptionq>3}
	There exists a TIpGM $\mu_{\mathbf{\tilde{h}}}:\Omega\to\mathbb{Q}_p$ associated with a boundary function  $\mathbf{\tilde{h}}=(\tilde{h}_1,\cdots, \tilde{h}_{q})$ if and only if $\tilde{h}_j=\ln_p(hz_j)$ for all $1\leq j\leq q-1$ and $\tilde{h}_q=\ln_p(h)$ where $h\in\mathcal{E}_p$ is any $p$-adic number and $\mathbf{z}=(z_1,\cdots z_{q-1})\in \mathcal{E}_p^{q-1}$ is defined either one of the following form
	\begin{itemize}
		\item[(\textbf{A})] $\mathbf{z}=\left(1,\cdots,1\right);$
		\item[(\textbf{B})] $\mathbf{z}=\left(z,\cdots,z\right)$ where $z\in\mathcal{E}_p\setminus\{1\}$ is a root of the following cubic equation 
		\begin{multline*}
		\quad \quad \quad (q-1)^3z^3+\left(3(q-1)^2-(\theta-1)^2(\theta+3(q-1)-1)\right)z^2\\
		+\left(3(q-1)-(\theta-1)^2(\theta+2)\right)z+1=0;
		\end{multline*}
		\item[(\textbf{C})] $\mathbf{z}=\mathbf{e}_{\alpha_1}+z\mathbf{e}_{\alpha_2}$ with $|\alpha_i|=m_i, \ m_1+m_2=q-1$ such that  $z\in\mathcal{E}_p\setminus\{1\}$ is a root of the following cubic equation
		\begin{multline*}
		\quad \quad \quad	m_2^3z^3+\left[3m_2^2(m_1+1)-(\theta-1)^2(\theta+3m_2-1)\right]z^2\\
		+\left[3m_2(m_1+1)^2-(\theta-1)^2(\theta+3(m_1+1)-1)\right]z+(m_1+1)^3=0;
		\end{multline*}
		\item[(\textbf{D})] $\mathbf{z}=z_1\mathbf{e}_{\alpha_1}+z_2\mathbf{e}_{\alpha_2}$  with $|\alpha_i|=m_i, \ m_1+m_2=q-1$ such that  
		\begin{itemize}
			\item[(i)] 	If  $m_1\neq \frac{1-\theta}{3}$ then 
			$z_1=-\frac{(\theta-1+3m_2)z_2+\theta+2}{(\theta-1+3m_1)}\in\mathcal{E}_p\setminus\{1\}$
			and $z_2\in\mathcal{E}_p\setminus\{1\}$ is a root of the following cubic equation 
			\begin{multline*}
			\quad \quad \quad\quad \quad \left[(m_1-m_2)z+(m_1-1)\right]^3\\
			+(\theta-1+3m_1)^2\left[(\theta-1+3m_2)z^2+(\theta+2)z\right]=0;
			\end{multline*}
			\item[(ii)] If  $m_2\neq \frac{1-\theta}{3}$ then 
			$z_2=-\frac{(\theta-1+3m_1)z_1+\theta+2}{(\theta-1+3m_2)}\in\mathcal{E}_p\setminus\{1\}$
			and $z_1\in\mathcal{E}_p\setminus\{1\}$ is a root of the following cubic equation 
			\begin{multline*}
			\quad \quad \quad\quad \quad \left[(m_2-m_1)z+(m_2-1)\right]^3\\
			+(\theta-1+3m_2)^2\left[(\theta-1+3m_1)z^2+(\theta+2)z\right]=0;
			\end{multline*}
		\end{itemize}
		\item[(\textbf{E})] $\mathbf{z}=z_1\mathbf{e}_{\alpha_1}+z_2\mathbf{e}_{\alpha_2}+\mathbf{e}_{\alpha_3}$ with $|\alpha_i|=m_i, \ m_1+m_2+m_3=q-1$ such that 
		\begin{itemize}
			\item[(i)] If $m_1\neq \frac{1-\theta}{3}$ then 
			$z_1=-\frac{(\theta-1+3m_2)z_2+3m_3+\theta+2}{(\theta-1+3m_1)}\in\mathcal{E}_p\setminus\{1\}$
			and $z_2\in\mathcal{E}_p\setminus\{1\}$ is a root of the following cubic equation
			\begin{multline*}
			\quad \quad \quad\quad \quad \left[(m_1-m_2)z+(m_1-m_3-1)\right]^3+(\theta-1+ 3m_1)^2(\theta-1+3m_2)z^2\\
			+(\theta-1+3m_1)^2(3m_3+\theta+2)z=0;
			\end{multline*}
			\item[(ii)]  If  $m_2\neq \frac{1-\theta}{3}$ then 
			$z_2=-\frac{(\theta-1+3m_1)z_1+3m_3+\theta+2}{(\theta-1+3m_2)}\in\mathcal{E}_p\setminus\{1\}$
			and $z_1\in\mathcal{E}_p\setminus\{1\}$ is a root of the following cubic equation 
			\begin{multline*}
			\quad \quad \quad\quad \quad	\left[(m_2-m_1)z+(m_2-m_3-1)\right]^3+(\theta-1+3m_2)^2(\theta-1+3m_1)z^2\\
			+(\theta-1+3m_2)^2(3m_3+\theta+2)z=0;	
			\end{multline*}
			\item[(iii)] If $\theta=1-q$ and $m_1=m_2=m_3+1$ then either $z_1\in\mathcal{E}_p\setminus\{1\}$ or $z_2\in\mathcal{E}_p\setminus\{1\}$ is any $p$-adic number  so that the second one is a root of the cubic equation $(z_1+z_2+1)^3=27z_1z_2$.
		\end{itemize}
	\end{itemize}
\end{theorem}

\section{\texorpdfstring{$H_m$-periodic $p$}{Lg}-adic Gibbs measures}

Let $m$ be some fixed natural number. A boundary function $\mathbf{{\tilde{h}}}:V\to\bq_p^{q}$ is called \textit{$m$-periodic} if $\mathbf{{\tilde{h}}}(x)=\mathbf{\tilde{h}}^{(j)}$ whenever $d(x,x^{(0)})\equiv j \pmod m$  with $0\leq j < m.$  A $p$-adic Gibbs measure associated with an $m$-periodic boundary function  is called  \textit{$H_m$-periodic.} In this case, the condition \eqref{systemofequationforh} takes the form $\mathbf{h}^{(j)}=k\mathbf{F}(\mathbf{h}^{(j+1)})$ for all $0\leq j < m$ with $\mathbf{h}^{(m)}:=\mathbf{h}^{(0)}$, or equivalently
$$
h^{(j)}_i=\ln_p\left(\frac{(\theta-1)\exp_p(h^{(j+1)}_i)+\sum_{l=1}^{q-1}\exp_p(h^{(j+1)}_l)+1}{\theta+ \sum_{l=1}^{q-1}\exp_p(h^{(j+1)}_l)}\right)^k, \quad 1\leq i \leq q-1.
$$ 
 Let $\mathbf{z}^{(j)}=(z^{(j)}_1,\cdots,z^{(j)}_{q-1})\in\mathcal{E}_p^{q-1}$ such that $z^{(j)}_i=\exp_p(h^{(j)}_i)$ for any $1 \leq i \leq q-1.$ In what follows, we write $\mathbf{z}^{(j)}=\exp_p(\mathbf{h}^{(j)}).$ Hence, we have the following result.
 
\begin{proposition}\label{CreteriaHmperiodicpadicGibbsmeasure}
There exists an $H_m$-periodic $p$-adic Gibbs measure if and only if the following system of equations 
 		\begin{equation}\label{equationwrtzperiodic}
 		z^{(j)}_i=\left(\frac{(\theta-1)z^{(j+1)}_i+\sum_{l=1}^{q-1}z^{(j+1)}_l+1}{\theta+ \sum_{l=1}^{q-1}z^{(j+1)}_l}\right)^k, \ \ 1\leq i \leq q-1, \ \ 0\leq j \leq m-1
 		\end{equation}
 		has a solution $\{\mathbf{z}^{(j)}\}_{j=0}^{m-1}\subset\mathcal{E}_p^{q-1}$ in which $\mathbf{z}^{(m)}:=\mathbf{z}^{(0)}.$
 \end{proposition}

We are looking for solutions of the form $\mathbf{z}^{(j)}=z^{(j)}\mathbf{e}_{\alpha}+\mathbf{e}_{\beta}$ with $|\alpha|+|\beta|=q-1$ of the system of equations \eqref{equationwrtzperiodic}. In this case, the system of equations \eqref{equationwrtzperiodic} takes the following form
\begin{eqnarray}
z^{(j)}=\left(\frac{(\theta+|\alpha|-1)z^{(j+1)}+q-|\alpha|}{|\alpha|z^{(j+1)}+\theta+ q-|\alpha|-1}\right)^k, \quad 0\leq j \leq m-1
\end{eqnarray}
with $z^{(m)}:=z^{(0)}.$ This means that $\{z^{(j)}\}_{j=0}^{m-1}$ is an $m$-cycle of \textit{the Potts--Bethe mapping} so-called
\begin{eqnarray}\label{Gthetaalphaqk}
G_{\theta,\alpha,q,k}(z)=\left(\frac{(\theta+|\alpha|-1)z+q-|\alpha|}{|\alpha|z+\theta+ q-|\alpha|-1}\right)^k.
\end{eqnarray}

If $\theta=1$ or $\theta=1-q$ then the Potts--Bethe mapping becomes a constant function. Throughout this paper, we suppose that $0<|\theta-1|_p<|q|_p<1$. This is one of the assumptions to ensure the non-unique translation-invariant $p$-adic Gibbs measures. We may assume that $1\leq |\alpha| \leq q-1$ otherwise if $|\alpha|=0$ then we obtain a trivial solution $\mathbf{z}^{(j)}=(1,1,\cdots, 1)$ for all $0\leq j \leq m-1$ of the system of equations \eqref{equationwrtzperiodic}. Thus, we have
\begin{eqnarray*}\label{Gthetaqk}
G_{\theta,\alpha,q,k}(z)=\left(\frac{\left(\frac{\theta-1}{|\alpha|}+1\right)z+\frac{q}{|\alpha|}-1}{z+\frac{\theta-1+ q}{|\alpha|}-1}\right)^k.
\end{eqnarray*}
If we set $\Theta=\frac{\theta-1}{|\alpha|}+1$ and $\mathsf{Q}=\frac{q}{|\alpha|},$ then we obtain
\begin{eqnarray*}\label{GThetaQk}
	f_{\Theta,\mathsf{Q},k}(z)=\left(\frac{\Theta z+\mathsf{Q}-1}{ z+\Theta+\mathsf{Q}-2}\right)^k.
\end{eqnarray*}
It is obvious that the conditions $\theta\neq 1$ and $\theta\neq 1-q$ are equivalent to the conditions $\Theta\neq 1$ and $\Theta\neq 1-\mathsf{Q}.$ Moreover, the condition $|\theta-1|_p<|q|_p$ (resp. $|\theta-1|_p=|q|_p$ or $|\theta-1|_p>|q|_p$)  is equivalent to the condition $|\Theta-1|_p<|\mathsf{Q}|_p$ (resp. $|\Theta-1|_p=|\mathsf{Q}|_p$ or $|\Theta-1|_p>|\mathsf{Q}|_p$). Therefore, without loss of generality, we assume that $|\alpha|=1$ and $q\in\mathbb{Q}_p.$ We are now ready to formulate the main results of this paper.

Let $f_{\theta, q, k} : \mathbb{Q}_p \to \mathbb{Q}_p$ be the Potts--Bethe mapping defined as
\begin{equation}\label{PottsBethe}
f_{\theta, q, k}(x)=\left(\frac{\theta x+q-1}{x+\theta+q-2}\right)^k
\end{equation}
where $\theta\in\mathcal{E}_p,$ $q\in\mathbb{Q}_p$ such that $0<|\theta-1|_p<|q|_p<1$. 

Denote by $\textup{\textbf{Dom}}\{f\}$ the domain of a mapping $f: \mathbb{Q}_p \to \mathbb{Q}_p$ and by $\mathbf{Fix}\{f\}$ the set of its fixed points. Then $\textup{\textbf{Dom}}\{f_{\theta,q,k}\}=\mathbb{Q}_p\setminus \{\mathbf{x}^{(\infty)}\}$ where $\mathbf{x}^{(\infty)}=2-\theta-q=\mathbf{x}^{(0)}-q-(\theta-1)$ and $\mathbf{x}^{(0)}=1\in \mathbf{Fix}\{f_{\theta, q, k}\}$ is always a fixed point for any $k\in\mathbb{N}$.

The dynamics of the Potts--Bethe mapping \eqref{PottsBethe} for the cases $k=1$ with $\theta, q\in\mathbb{Q}_p$ and $k=2$ with $0<|\theta-1|_p<|q|^2_p<1$ were studied in \cite{Liao2} and \cite{FMOK2016}, respectively. Here, we study its dynamics in the case $k=3$ with $p\equiv 1 \ (\rm{mod} \ 3)$ and $p=2,3$. The case $k=3$ with $p \geq 5,\ p\equiv 2 \ (\rm{mod} \ 3)$ has been studied in \cite{SMAA2016Submitted}. 

Let us consider the following cubic equation
\begin{equation}\label{cubiceqwrty}
y^3-(1+\theta+\theta^2)y^2-(2\theta+1)(1-\theta-q)y-(1-\theta-q)^2=0.
\end{equation}

This cubic equation always has three roots $\mathbf{y}^{(1)},\mathbf{y}^{(2)},\mathbf{y}^{(3)}\in\mathbb{Z}_p$ (see Proposition \ref{structureroot}). Let  $\mathbf{x}^{(0)}=1,\ \mathbf{x}^{(i)}=\mathbf{x}^{(0)}-q+(\theta-1)(\mathbf{y}^{(i)}-1)$ for $1\leq i\leq3$. We introduce the following sets, 
\begin{eqnarray*}
\mathcal{A}_0 &:=& \left\{ x \in \mathbb{Q}_p : \left|x-\mathbf{x}^{(0)}\right|_p < |q|_p \right\},\\
\mathcal{C}_{1} &:=& \left\{ x \in \mathbb{Q}_p : \left|x-\mathbf{x}^{(1)}\right|_p < \left|x-\mathbf{x}^{(\infty)}\right|_p = |\theta-1|_p \right\}, \\
\mathcal{C}_{2} &:=& \left\{ x \in \mathbb{Q}_p : \left|x-\mathbf{x}^{(2)}\right|_p < \left|x-\mathbf{x}^{(\infty)}\right|_p = |q|_p|\theta-1|_p\right\},\\
\mathcal{C}_{3} &:=& \left\{ x \in \mathbb{Q}_p : \left|x-\mathbf{x}^{(3)}\right|_p < \left|x-\mathbf{x}^{(\infty)}\right|_p = |q|_p|\theta-1|_p\right\}.
\end{eqnarray*}
The attracting basin of the fixed point $\mathbf{x}^{(0)}$ is defined as
$$\mathfrak{B}(\mathbf{x}^{(0)}):=\left\{ x \in \mathbb{Q}_p : \lim_{n \to +\infty}f^{n}_{\theta, q, 3}(x)=\mathbf{x}^{(0)}\right\}$$
where $f_{\theta, q, 3}^{n}(x)$ for $n\in\mathbb{N}$ is the $n$-th iteration of $f_{\theta, q, 3}$. The following theorem is our main theorem. 

\begin{theorem}\label{PottsBetheMappingDynamics}
Let $k=3,\ p \equiv 1 \ (\rm{mod} \ 3),\ $ and $ 0 <|\theta-1|_p<|q|_p<1$.  Then the following statements hold.
\begin{itemize}
	\item[(i)] One has $\mathbf{Fix}\{f_{\theta, q, 3}\}=\{\mathbf{x}^{(0)}, \ \mathbf{x}^{(1)},\ \mathbf{x}^{(2)},\ \mathbf{x}^{(3)}\}$ in which $\mathbf{x}^{(0)}$ is attracting  and $\mathbf{x}^{(1)},\ \mathbf{x}^{(2)}$ and $\mathbf{x}^{(3)}$ are repelling;
    \item[(ii)] One has $$\textup{\textbf{Dom}}\{f_{\theta,q,3}\}\setminus \left(\mathcal{C}_1 \cup \mathcal{C}_{2}\cup \mathcal{C}_{3}\right) \subset \mathfrak{B}\left( \mathbf{x}^{(0)} \right)=\bigcup\limits^{+\infty}_{n = 0}f^{-n}_{\theta, q, 3}\left(\mathcal{A}_0\right);$$
    \item[(iii)] There exists $\mathcal{J}\subset \mathcal{C}_1 \cup \mathcal{C}_{2}\cup \mathcal{C}_{3}$ such that  the dynamical system $f_{\theta, q, 3}:\mathcal{J}\to\mathcal{J}$ is topologically conjugate to the full shift dynamics $(\Sigma_3,\sigma)$ on three symbols;
    \item[(iv)] One has $$\mathbb{Q}_p=\mathfrak{B}\left( \mathbf{x}^{(0)} \right)\cup \mathcal{J}\cup \bigcup\limits_{n=0}^{+\infty}f_{\theta, q, 3}^{-n}\{\mathbf{x}^{(\infty)}\}.$$
\end{itemize}
\end{theorem}
The proof of Theorem \ref{PottsBetheMappingDynamics} will be given in Section 4. As an application of Theorem \ref{PottsBetheMappingDynamics}, we have the following theorem which shows the vastness of periodic $p$-adic Gibbs measures.

\begin{theorem}\label{HmperiodicpadicGibbsMeasure}
Let $k=3,\ p \equiv 1 \ (\rm{mod} \ 3)$. Suppose $0 < |\theta-1|_p < |q|_p<1$. Write $q=\frac{q^{*}}{|q|_p}$ with $q^{*}\centernot\mid p.$  Then there exist $H_m$-periodic $p$-adic Gibbs measures associated with the $m$-periodic boundary functions $\{\mathbf{\tilde{h}}^{(j)}\}_{j=0}^{m-1}$ where $$\mathbf{\tilde{h}}^{(j)}=(\ln_p(hz^{(j)}_1),\cdots, \ln_p(hz^{(j)}_{q-1}), \log_p(h))$$ with $h\in\mathcal{E}_p,\ \mathbf{z}^{(j)}=(z^{(j)}_1,\cdots z^{(j)}_{q-1})\in \mathcal{E}_p^{q-1}$ such that $\mathbf{z}^{(j)}=z^{(j)}\mathbf{e}_{\alpha}+\mathbf{e}_{\beta}, |\alpha|+|\beta|=q-1,$ $|\alpha|\not\in\{ \frac{1}{|q|_p}, \frac{2}{|q|_p},\cdots, \frac{q^{*}-1}{|q|_p}\} $ and $\{z^{(j)}\}_{j=0}^{m-1}$ is an $m$-cycle of the Potts--Bethe mapping
\begin{eqnarray}\label{Gthetaalphaq3}
G_{\theta,\alpha,q,3}(z)=\left(\frac{(\theta+|\alpha|-1)z+q-|\alpha|}{|\alpha|z+\theta+ q-|\alpha|-1}\right)^3.
\end{eqnarray}
\end{theorem}

\begin{proof}
As we have already mentioned (see Proposition \ref{CreteriaHmperiodicpadicGibbsmeasure}) that there exists an $H_m$-periodic $p$-adic Gibbs measure associated with an $m$-periodic boundary  function $\{\mathbf{\tilde{h}}^{(j)}\}_{j=0}^{m-1},$ $\mathbf{\tilde{h}}^{(j)}=(\ln_p(hz^{(j)}_1),\cdots, \ln_p(hz^{(j)}_{q-1}), \ln_p(h))$ with $h\in\mathcal{E}_p,\ \mathbf{z}^{(j)}=(z^{(j)}_1,\cdots z^{(j)}_{q-1})\in \mathcal{E}_p^{q-1}$ of the form $\mathbf{z}^{(j)}=z^{(j)}\mathbf{e}_{\alpha}+\mathbf{e}_{\beta}, \ |\alpha|+|\beta|=q-1$ if and only if $\{z^{(j)}\}_{j=0}^{m-1}\subset \mathcal{E}_p$ is an $m$-cycle of the Potts--Bethe mapping \eqref{Gthetaalphaq3} or equivalently an $m$-cycle of the following mapping
	\begin{eqnarray}\label{fThetaQk}
		f_{\Theta,\mathsf{Q},3}(z)=\left(\frac{\Theta z+\mathsf{Q}-1}{z+\Theta+\mathsf{Q}-2}\right)^3
	\end{eqnarray}
where $\Theta:=\frac{\theta-1}{|\alpha|}+1$ and $\mathsf{Q}:=\frac{q}{|\alpha|}.$ Due to Theorem \ref{PottsBetheMappingDynamics} $(iii)$, if $0<|\Theta-1|_p<|\mathsf{Q}|_p<1$	then the Potts--Bethe mapping \eqref{fThetaQk} has an $m$-cycle $\{z^{(j)}\}_{j=0}^{m-1}\subset \mathcal{C}_1 \cup \mathcal{C}_{2}\cup \mathcal{C}_{3} \subset \mathcal{E}_p.$ 

We now describe all possible sets $\alpha\subset\mathbf{I}_{q-1}$ for which one has that $0<|\Theta-1|_p<|\mathsf{Q}|_p<1.$ Let $q=p^l \cdot q^{*}=\frac{q^{*}}{|q|_p}$ such that $q^{*}\centernot\mid p, \ l\in\mathbb{N}.$ One has $|\mathsf{Q}|_p=|\frac{q}{|\alpha|}|_p<1$ if and only if $|\alpha|\not\in \{p^l, 2p^l,\cdots, (q^{*}-1)p^l \}=\{ \frac{1}{|q|_p}, \frac{2}{|q|_p},\cdots, \frac{q^{*}-1}{|q|_p}\}.$ Moreover, since $|\theta-1|_p<|q|_p$ and $|\alpha|\in\mathbb{N},$ we get $|\Theta-1|_p<|\mathsf{Q}|_p.$
\end{proof}

\begin{remark}
Let $\mathcal{HGM}(m,q:3)$ be the set of $H_m$-periodic $p$-adic Gibbs measures of the $q$-states Potts model on the Cayley tree of order three. It follows from Theorem \ref{HmperiodicpadicGibbsMeasure} that $$|\mathcal{HGM}(m,q:3)|\geq \left(2^{q-1}-\sum\limits_{i=1}^{q^{*}-1}{q-1\choose{|q|^{-1}_pi}}\right)\cdot 3^m$$ where $q=\frac{q^{*}}{|q|_p}$ with $q^{*}\centernot\mid p.$
\end{remark}\medskip
	
\section{Dynamics of the Potts--Bethe mapping for \texorpdfstring{$k=3$ and $p\equiv 1 \ (\rm{mod} \ 3)$}{Lg}} \label{section4}

We investigate the dynamics of the Potts--Bethe mapping $f_{\theta, q, 3} : \textup{\textbf{Dom}}\{f_{\theta,q,3}\} \to \mathbb{Q}_p$
\begin{equation}\label{IsingPottsk=3}
f_{\theta, q, 3}(x)=\left(\frac{\theta x+q-1}{x+\theta+q-2}\right)^3
\end{equation}
where $\textup{\textbf{Dom}}\{f_{\theta,q,3}\}=\mathbb{Q}_p\setminus \{\mathbf{x}^{(\infty)}\}$ and $\mathbf{x}^{(\infty)}=2-\theta-q$. 

Throughout this section, we assume that $0<|\theta-1|_p<|q|_p<1$ and $p\equiv 1 \ (\rm{mod} \ 3)$. 

\subsection{Fixed points}
It is clear that $\mathbf{x}^{(0)}=1 \in \mathbf{Fix}\{f_{\theta, q, 3}\}$. Moreover, it follows from $f_{\theta, q, 3}(x)-1=x-1$ that
\begin{multline*}
(x-1)(\theta-1)\frac{(\theta x+q-1)^2+(\theta x+q-1)(x+\theta+q-2)+(x+\theta+q-2)^2}{(x+\theta+q-2)^3}\\
=(x-1).	
\end{multline*}
Therefore, any other fixed point $x \neq \mathbf{x}^{(0)}$ is a root of the following cubic equation
\begin{equation}\label{wrtx}
(\theta-1)(\theta x+q-1)\left((\theta+1)x+\theta+2q-3\right)= (x+q-1)(x+\theta+q-2)^2.
\end{equation}

We introduce a new variable $y:=\frac{x-1+q}{\theta-1}+1$. Then the cubic equation \eqref{wrtx} can be written with respect to $y$ as follows
\begin{equation}\label{wrty}
y^3-(1+\theta+\theta^2)y^2-(2\theta+1)(1-\theta-q)y-(1-\theta-q)^2=0.
\end{equation}

Let us find all possible roots of the cubic equation \eqref{wrty}.

\begin{proposition}\label{structureroot}
Write $q=p^rq^*$ for some $r\in\mathbb{N}$ with $q^*=q_0+q_1p+\dots\in\mathbb{Z}^*_p.$ Then the cubic equation \eqref{wrty} always has three roots $\mathbf{y}^{(1)},\mathbf{y}^{(2)},\mathbf{y}^{(3)}$ such that
\begin{itemize}
	\item[(i)] $|\mathbf{y}^{(1)}|_p=1>|1-\theta-q|_p=|\mathbf{y}^{(2)}|_p=|\mathbf{y}^{(3)}|_p;$
	\item[(ii)] $\mathbf{y}^{(1)}=3+(p-q_0)p^r+\cdots$ or equivalently $|\mathbf{y}^{(1)}-3+q|_p<|\mathbf{y}^{(1)}-3|_p=|q|_p;$
	\item[(iii)] $\mathbf{y}^{(2)} = \frac{1}{\left|q\right|_p} \left(y^{(2)}_0+y^{(2)}_1p+\dots\right), \	\mathbf{y}^{(3)} = \frac{1}{\left|q\right|_p} \left(y^{(3)}_0+y^{(3)}_1p+\dots\right)$
	where $y^{(2)}_0$ and $y^{(3)}_0$ are roots of the following congruent equation
	\begin{equation}\label{congruenty2y3}
	3t^2+3(1-\theta-q)^*t+\left((1-\theta-q)^*\right)^2 \equiv 0 \ (\rm{mod} \ p).
	\end{equation}
\end{itemize}
\end{proposition}
\begin{remark}
	The congruent equation \eqref{congruenty2y3} has discriminant $D=-3\left((1-\theta-q)^{*}\right)^2$. It is solvable if and only $D$ (or equivalently $-3$) is a quadratic residue modulo $p$. $-3$ is a quadratic residue modulo $p$ if and only if $p \equiv 1 \ (\rm{mod} \ 3)$.
\end{remark}
\begin{proof}
We study the solvability of the cubic equation \eqref{wrty} over $\mathbb{Q}_p$. We first show that if the cubic equation \eqref{wrty} has any root in the $p$-adic field $\mathbb{Q}_p$ then it must lie in the set $\mathbb{Z}_p.$ Suppose the contrary, i.e. the cubic equation \eqref{wrty} has a root $y$ such that $|y|_p>1.$ Since $p\equiv 1 \ (mod \ 3)$ and $y^3=(1+\theta+\theta^2)y^2+(2\theta+1)(1-\theta-q)y+(1-\theta-q)^2,$ then 
$$|y|^3_p=|(1+\theta+\theta^2)y^2+(2\theta+1)(1-\theta-q)y+(1-\theta-q)^2|_p=|y|^2_p.$$ It is a contradiction. Therefore, any root of the cubic equation \eqref{wrty} must lie in the set $\mathbb{Z}_p.$ We refer to \cite{SMAA2015d} for the detailed study of the general cubic equation over $\mathbb{Z}_p$.

Let $a=-(1+\theta+\theta^2),\ b=-(2\theta+1)(1-\theta-q)$ and $c=-(1-\theta-q)^2$. One can verify that $|b|_p^2=|c|_p=|1-\theta-q|_p^2<1=|a|_p$ and $|b|_p<|a|_p^2,\ |c|_p<|a|_p^3$. Thus, by Theorem 5.1 of \cite{SMAA2015d}, the cubic equation \eqref{wrty} always has a root $\mathbf{y}^{(1)}$ for which $|\mathbf{y}^{(1)}|_p=|a|_p=1$. Moreover, it follows from $|\theta-1|_p<|q|_p<1$ and $\mathbf{y}^{(1)} = y_0^{(1)}+y_1^{(1)}p+\dots$ that $\left(y_0^{(1)}\right)^3-3\left(y_0^{(1)}\right)^2 \equiv 0 \ (mod \ p).$ Thus, $y_0^{(1)}=3$ and $\mathbf{y}^{(1)}=3+y_1^{(1)}p+\dots$.

Let $\delta_1=b^2-4ac$. It is clear that $\delta_1=-3(1-\theta-q)^2$ and $|b|_p^2=|a|_p|c|_p=|1-\theta-q|_p^2=|\delta_1|_p$. Since $p\equiv 1 \ (mod \ 3),$ there always exists $\sqrt{\delta_1}\in\mathbb{Q}_p$ (or equivalently there exists $\sqrt{-3}\in\mathbb{Q}_p$). Therefore, by Theorem 5.1 of \cite{SMAA2015d}, the cubic equation \eqref{wrty} has two more roots $\mathbf{y}^{(2)}$ and $\mathbf{y}^{(3)}$ such that $|\mathbf{y}^{(2)}|_p=|\mathbf{y}^{(3)}|_p=\frac{|b|_p}{|a|_p}=|1-\theta-q|_p<1$. 

We know that $|1-\theta-q|_p=|q|_p=p^{-r}$ for some $r\in\mathbb{N}$. Then, $\mathbf{y}^{(i)}=p^r(\mathbf{y}^{(i)})^*$ for $i=2,3$ and $1-\theta-q=p^r(1-\theta-q)^{*}$. By plugging $\mathbf{y}^{(i)}$ into the cubic equation \eqref{wrty}, we have
$$
p^r((\mathbf{y}^{(i)})^{*})^3-(1+\theta+\theta^2)((\mathbf{y}^{(i)})^*)^2-(2\theta+1)(1-\theta-q)^*(\mathbf{y}^{(i)})^* -\left((1-\theta-q)^{*}\right)^2=0.
$$
Then for $i=2,3$
\begin{equation*}
3(y^{(i)}_0)^2+3(1-\theta-q)^*y^{(i)}_0+\left((1-\theta-q)^{*}\right)^2 \equiv 0 \ ({\rm mod} \ p).
\end{equation*}
Hence, $y^{(2)}_0$ and $y^{(3)}_0$ are roots of the congruent equation
\begin{equation*}
3t^2+3(1-\theta-q)^*t+\left((1-\theta-q)^*\right)^2 \equiv 0 \ ({\rm mod} \ p).
\end{equation*}
Now, to prove $\left|\mathbf{y}^{(1)}-3\right|_p=|q|_p>|\mathbf{y}^{(1)}-3+q|_p$, we introduce a new variable $z:=y-3$. Then the cubic equation \eqref{wrty} takes the following form with respect to $z$
\begin{multline}\label{wrtz}
z^3+\left[8-\theta-\theta^2\right]z^2+\left[21-6\theta(\theta+1)-(2\theta+1)(1-\theta-q)\right]z \\ +\left[q(4\theta+5)-(\theta-1)(4\theta+14)-q^2\right]=0.
\end{multline}
Let $A=(8-\theta-\theta^2),\ B=21-6\theta(\theta+1)-(2\theta+1)(1-\theta-q)$ and $C=q(4\theta+5)-(\theta-1)(4\theta+14)-q^2$. One can check that $|A|_p=1,\ |B|_p=1$ and $|C|_p=|q|_p<1$. Thus $|B|_p=|A|_p^2$ and $|C|_p<|A|_p^3$ which imply that the cubic equation \eqref{wrtz} always has a root $\mathbf{z}^{(1)}$ for which $|\mathbf{z}^{(1)}|_p=\frac{|C|_p}{|B|_p}=|q|_p$ (see the proof of Theorem 5.1, \cite{SMAA2015d}). Therefore, we have $|\mathbf{z}^{(1)}|_p=|\mathbf{y}^{(1)}-3|_p=|q|_p$ which means that $\mathbf{y}^{(1)}=3+y^{(1)}_rp^r+\cdots$ and we write the cubic equation \eqref{wrty} as follows
\begin{equation}\label{wrty2}
y^2(y-3)-(\theta+2)(\theta-1)y^2-(2\theta+1)(1-\theta-q)y-(1-\theta-q)^2=0.
\end{equation}
Since $|\theta-1|_p<|q|_p<1,$ we have
\begin{eqnarray*}
\left(\mathbf{y}^{(1)}\right)^2&\equiv& 9+6y^{(1)}_rp^r \ ({\rm mod} \ p^{r+1}), \\
\theta&\equiv & 1 \ ({\rm mod} \ p^{r+1}).
\end{eqnarray*}
It follows from  \eqref{wrty2}
\begin{eqnarray*}
9y^{(1)}_rp^r+9p^rq_0 \equiv 0 \ ({\rm mod} \ p^{r+1}), \ \textup{equivalently} \quad y^{(1)}_r+q_0 \equiv 0 \ ({\rm mod} \ p).
\end{eqnarray*}
Therefore, $y^{(1)}_r=p-q_0 \equiv -q_0 \ (mod \ p),$ which implies $|\mathbf{y}^{(1)}-3+q|_p<|\mathbf{y}^{(1)}-3|_p.$
\end{proof}

\begin{proposition}\label{fixedpoints}
$\mathbf{Fix}\{f_{\theta, q, 3}\}=\{\mathbf{x}^{(0)}, \ \mathbf{x}^{(1)},\ \mathbf{x}^{(2)},\ \mathbf{x}^{(3)}\}$, where $\mathbf{x}^{(0)}=1,\ \mathbf{x}^{(i)}=1-q+(\theta-1)(\mathbf{y}^{(i)}-1)$ for $1\leq i\leq3$ and  $\mathbf{y}^{(1)},\mathbf{y}^{(2)},\mathbf{y}^{(3)}$ are the roots of the cubic equation \eqref{wrty}.
\end{proposition}

\begin{proof}
The proof follows from Proposition \ref{structureroot}.
\end{proof}

\subsection{Local behavior of the fixed points}

Recall that a fixed point $\mathbf{x}$ of the Potts--Bethe mapping is called \textit{attracting} if $0 \leq |\lambda|_p < 1$, \textit{indifferent} if $|\lambda|_p=1$ and \textit{repelling} if $|\lambda|_p>1$ where $\lambda=f^{'}_{\theta,q,3}(\mathbf{x})$.

\begin{proposition}\label{behaviourfixedpoints}
The following statements hold:
\begin{itemize}
\item[(i)] $\mathbf{x}^{(0)}$ is an attracting;
\item[(ii)] $\mathbf{x}^{(1)},\ \mathbf{x}^{(2)}$ and $\mathbf{x}^{(3)}$ are repelling.
\end{itemize}
\end{proposition}

\begin{proof}
We have to show $\left|f'_{\theta, q, 3}(\mathbf{x}^{(0)})\right|_p<1$ and $\left|f'_{\theta, q, 3}(\mathbf{x}^{(i)})\right|_p>1$ for $1\leq i\leq 3$. It is easy to check that for $0\leq i\leq 3$,
\begin{eqnarray}\label{derivative}
f'_{\theta, q, 3}(\mathbf{x}^{(i)})&=&\frac{3(\theta-1)(\theta-1+q)(\theta \mathbf{x}^{(i)}+q-1)^2}{(\mathbf{x}^{(i)}+\theta+q-2)^4}\nonumber\\
&=&\frac{3(\theta-1)(\theta-1+q)\mathbf{x}^{(i)}}{(\theta \mathbf{x}^{(i)}+q-1)(\mathbf{x}^{(i)}+\theta+q-2)}.
\end{eqnarray}
Hence, $\left|f'_{\theta, q, 3}(\mathbf{x}^{(0)})\right|_p=\frac{|\theta-1|_p}{|q|_p} < 1$.
Recalling $\mathbf{x}^{(i)}=1-q+(\theta-1)(\mathbf{y}^{(i)}-1)$ for $1 \leq i \leq 3$, we deduce from \eqref{derivative} that
\begin{equation}\label{derivativefixed2}
f'_{\theta, q, 3}(\mathbf{x}^{(i)})=\frac{3(\theta-1+q)\mathbf{x}^{(i)}}{(\theta-1)\mathbf{y}^{(i)}(\theta \mathbf{y}^{(i)}+1-\theta-q)}.
\end{equation}
Due to Proposition \ref{structureroot}, we have $\left|\mathbf{y}^{(2)}\right|_p=\left|\mathbf{y}^{(3)}\right|_p=|1-\theta-q|_p<1=\left|\mathbf{y}^{(1)}\right|_p$ and $|\mathbf{x}^{(i)}|_p=1$ for $1 \leq i \leq 3$. Therefore, $\left|f'_{\theta, q, 3}(\mathbf{x}^{(1)})\right|_p=\frac{|q|_p}{|\theta-1|_p} > 1$. For the fixed points $\mathbf{x}^{(2)}$ and $\mathbf{x}^{(3)}$, it follows from \eqref{derivativefixed2} that for $i=2,3$
\begin{equation}\label{derivativefixed3}
f'_{\theta, q, 3}(\mathbf{x}^{(i)})=\frac{3(\theta-1+q)|1-\theta-q|_p^2\mathbf{x}^{(i)}}{(\theta-1)\left(\mathbf{y}^{(i)}\right)^*\left(\theta \left(\mathbf{y}^{(i)}\right)^{*}+\left(1-\theta-q\right)^{*}\right)}.
\end{equation}
By means of \eqref{congruenty2y3}, we can check $\theta\left(\mathbf{y}^{(i)}\right)^* \not\equiv -\left(1-\theta-q\right)^* \ (mod \ p)$ for $i=2,3$. Consequently, for $ i=2,3$
\begin{eqnarray*}
\left|f'_{\theta, q, 3}(\mathbf{x}^{(i)})\right|_p=\frac{|\theta-1+q|_p}{|1-\theta-q|_p^2|\theta-1|_p}=\frac{1}{|q|_p|\theta-1|_p} > 1.
\end{eqnarray*}
\end{proof}

\subsection{Attracting basin of the attracting fixed point}
We describe the attracting basin $$\mathfrak{B}(\mathbf{x}^{(0)}):=\left\{ x \in \mathbb{Q}_p : \lim_{n \to +\infty}f^{n}_{\theta, q, 3}(x)=\mathbf{x}^{(0)}\right\}$$
of the attracting fixed point $\mathbf{x}^{(0)}=1.$ 

We introduce the following sets
\begin{eqnarray*}
\mathcal{A}_0 &:=& \left\{ x \in \mathbb{Q}_p : \left|x-\mathbf{x}^{(0)}\right|_p < |q|_p \right\}, \\ 
\mathcal{A}_1 &:=& \left\{ x \in \mathbb{Q}_p : \left|x-\mathbf{x}^{(0)}\right|_p > |q|_p \right\}, \\ 
\mathcal{A}_{0,\infty} &:=& \left\{ x \in \mathbb{Q}_p : \left|x-\mathbf{x}^{(\infty)}\right|_p = \left|x-\mathbf{x}^{(0)}\right|_p = |q|_p \right\}, \\ 
\mathcal{A}_2 &:=& \left\{ x \in \mathbb{Q}_p : \left|\theta-1\right|_p < \left|x-\mathbf{x}^{(\infty)}\right|_p < |q|_p \right\},\\
\mathcal{A}_{1,\infty} &:=& \left\{ x \in \mathbb{Q}_p : \left|x-\mathbf{x}^{(1)}\right|_p = \left|x-\mathbf{x}^{(\infty)}\right|_p = |\theta-1|_p \right\}, \\ 
\mathcal{C}_{1} &:=& \left\{ x \in \mathbb{Q}_p : \left|x-\mathbf{x}^{(1)}\right|_p < \left|x-\mathbf{x}^{(\infty)}\right|_p = |\theta-1|_p \right\}, \\ 
\mathcal{A}_3 &:=& \left\{x \in \mathbb{Q}_p : |q|_p\left|\theta-1\right|_p < \left|x-\mathbf{x}^{(\infty)}\right|_p < \left|\theta-1\right|_p \right\}, \\
\mathcal{A}_{2,3,\infty} &:=& \left\{ x \in \mathbb{Q}_p : \left|x-\mathbf{x}^{(2)}\right|_p =\left|x-\mathbf{x}^{(3)}\right|_p = \left|x-\mathbf{x}^{(\infty)}\right|_p = |q|_p|\theta-1|_p\right\}, \\
\mathcal{C}_{2} &:=& \left\{ x \in \mathbb{Q}_p : \left|x-\mathbf{x}^{(2)}\right|_p < \left|x-\mathbf{x}^{(\infty)}\right|_p = |q|_p|\theta-1|_p\right\},\\
\mathcal{C}_{3} &:=& \left\{ x \in \mathbb{Q}_p : \left|x-\mathbf{x}^{(3)}\right|_p < \left|x-\mathbf{x}^{(\infty)}\right|_p = |q|_p|\theta-1|_p\right\},\\
\mathcal{A}_{\infty} &:=& \left\{ x \in \mathbb{Q}_p : \ 0<\left|x-\mathbf{x}^{(\infty)}\right|_p < |q|_p\left|\theta-1\right|_p \right\}.
\end{eqnarray*}
We have
$$\textup{\textbf{Dom}}\{f_{\theta,q,3}\}=\mathcal{A}_0 \cup \mathcal{A}_1\cup \mathcal{A}_2\cup \mathcal{A}_3\cup \mathcal{A}_{0,\infty} \cup \mathcal{A}_{1,\infty}\cup \mathcal{A}_{2,3,\infty}\cup\mathcal{A}_{\infty}\cup \mathcal{C}_{1}\cup \mathcal{C}_{2}\cup \mathcal{C}_{3}.$$

The following properties of the fixed points $\mathbf{x}^{(0)}, \mathbf{x}^{(1)}, \mathbf{x}^{(2)}, \mathbf{x}^{(3)}$ will be often used.
\begin{itemize}
	\item[(i)] $\mathbf{x}^{(0)}=1$ and $\mathbf{x}^{(\infty)}=2-q-\theta=\mathbf{x}^{(0)}-q-(\theta-1);$
	\item[(ii)] $\mathbf{x}^{(i)}=1-q+(\theta-1)(\mathbf{y}^{(i)}-1)=\mathbf{x}^{(\infty)}+(\theta-1)\mathbf{y}^{(i)}$ for $i=1,2,3;$
	\item[(iii)] $|\mathbf{x}^{(\infty)}|_p=|\mathbf{x}^{(i)}|_p=1$ and $\mathbf{x}^{(\infty)},\mathbf{x}^{(i)}\in\mathcal{E}_p$ for $i=0,1,2,3;$
	\item[(iv)] $|\mathbf{y}^{(2)}|_p=|\mathbf{y}^{(3)}|_p=|q|_p<1=|\mathbf{y}^{(1)}|_p$ and $|\mathbf{y}^{(1)}-3+q|_p<|\mathbf{y}^{(1)}-3|_p=|q|_p;$
	\item[(v)] $|\mathbf{y}^{(2)}-\mathbf{y}^{(3)}|_p=|q|_p;$ 
	\item[(vi)] $|\mathbf{x}^{(0)}-\mathbf{x}^{(\infty)}|_p=|q|_p$ and $|\mathbf{x}^{(1)}-\mathbf{x}^{(\infty)}|_p=|\theta-1|_p;$
	\item[(vii)] $|\mathbf{x}^{(2)}-\mathbf{x}^{(\infty)}|_p=|\mathbf{x}^{(3)}-\mathbf{x}^{(\infty)}|_p=|q|_p|\theta-1|_p;$
	\item[(viii)] $|\mathbf{x}^{(i)}-\mathbf{x}^{(0)}|_p=|q|_p$ for $i=1,2,3$ and $|\mathbf{x}^{(i)}-\mathbf{x}^{(1)}|_p=|\theta-1|_p$ for $i=2,3;$  
	\item[(ix)] $|\mathbf{x}^{(2)}-\mathbf{x}^{(3)}|_p=|q|_p|\theta-1|_p.$
	\item[(x)] $\mathbf{x}^{(0)} \in \mathcal{A}_0,\ \mathbf{x}^{(1)} \in \mathcal{C}_1, \ \mathbf{x}^{(2)}\in\mathcal{C}_{2},\ \mathbf{x}^{(3)}\in\mathcal{C}_{3}.$ 
\end{itemize}

\begin{proposition}\label{p=1mod3}
The following inclusions hold:
\begin{itemize}
\item[(i)] $\mathcal{A}_0 \cup \mathcal{A}_1 \cup \mathcal{A}_2 \cup \mathcal{A}_{0,\infty}\subset f^{-1}_{\theta,q,3}\left(\mathcal{A}_0\right);$
\item[(ii)] $\mathcal{A}_{\infty}\cup\mathcal{A}_3\cup\mathcal{A}_{2,3,\infty}\subset f^{-1}_{\theta,q,3}\left(\mathcal{A}_{1} \right)\subset f^{-2}_{\theta,q,3}\left(\mathcal{A}_0\right);$
\item[(iii)] $\mathcal{A}_{1,\infty}\subset f^{-1}_{\theta,q,3}\left(\mathcal{A}_{0,\infty} \right)\subset f^{-2}_{\theta,q,3}\left(\mathcal{A}_0\right).$
\end{itemize}
\end{proposition}

\begin{proof}
(i) We have
\begin{equation}\label{fminusx0}
f_{\theta,q,3}(x)-\mathbf{x}^{(0)} =\frac{(\theta-1)(x-\mathbf{x}^{(0)})}{(x-\mathbf{x}^{(\infty)})^3}g(x)
\end{equation}
where $g(x):=(\theta x+q-1)^2+(\theta x+q-1)(x-\mathbf{x}^{(\infty)})+(x-\mathbf{x}^{(\infty)})^2$. 

Let us first show for any $x \in \mathcal{A}_0 \cup \mathcal{A}_1,\ f_{\theta,q,3}(x) \in \mathcal{A}_0$.
In fact
$$
|x-\mathbf{x}^{(\infty)}|_p=\left|x-\mathbf{x}^{(0)}+q+(\theta-1)\right|_p=
\begin{cases}
|q|_p, & x \in \mathcal{A}_0, \\
|x-\mathbf{x}^{(0)}|_p, & x \in \mathcal{A}_1, \\
\end{cases}
$$
$$
|\theta x+q-1|_p=\left|\theta(x-\mathbf{x}^{(0)})+q+(\theta-1)\right|_p=
\begin{cases}
|q|_p, & x \in \mathcal{A}_0, \\
|x-\mathbf{x}^{(0)}|_p, & x \in \mathcal{A}_1.
\end{cases}
$$
Thus,
$$|g(x)|_p \leq
\begin{cases}
|q|_p^2, & x \in \mathcal{A}_0, \\
|x-\mathbf{x}^{(0)}|_p^2, & x \in \mathcal{A}_1. \\
\end{cases}
$$
Therefore, we deduce from \eqref{fminusx0} that 
$$
|f_{\theta,q,3}(x)-\mathbf{x}^{(0)}|_p \leq
\left\{
\begin{array}{cc}
\frac{|x-\mathbf{x}^{(0)}|_p}{|q|_p}|\theta-1|_p, & x \in \mathcal{A}_0 \\
|\theta-1|_p, & x \in 
\mathcal{A}_1
\end{array}
\right\}
\leq |\theta-1|_p<|q|_p.
$$
This means $f_{\theta,q,3}(x) \in \mathcal{A}_0$ for any $x \in \mathcal{A}_0 \cup \mathcal{A}_1$.

Now, we show for any $x \in \mathcal{A}_{0,\infty},\ f_{\theta,q,3}(x) \in \mathcal{A}_0$. We have
\begin{eqnarray*}
|x-\mathbf{x}^{(0)}|_p=|x-\mathbf{x}^{(\infty)}|_p =|q|_p,\\	
|\theta x+q-1|_p = \left|\theta(x-\mathbf{x}^{(\infty)})+(\theta-1)(1-\theta-q)\right|_p= |q|_p,
\end{eqnarray*}
which imply $|g(x)|_p  \leq |q|^2_p$. Then by \eqref{fminusx0}, $|f_{\theta,q,3}(x)-\mathbf{x}^{(0)}|_p \leq {|\theta-1|_p}<|q|_p$ for any $x \in \mathcal{A}_{0,\infty}$.

Furthermore, we also have
\begin{eqnarray}\label{IsingPottsSingular}
f_{\theta, q, 3}(x)= \left(\theta+\frac{(\theta-1)(1-\theta-q)}{x-\mathbf{x}^{(\infty)}}\right)^3
\end{eqnarray}
and
\begin{eqnarray}\label{IsingPottsSingular2}
f_{\theta, q, 3}(x)-\mathbf{x}^{(0)}=\frac{(\theta-1)(x-\mathbf{x}^{(0)})}{x-\mathbf{x}^{(\infty)}}g_1(x)
\end{eqnarray}
where $g_1(x)=\left(\theta+\frac{(\theta-1)(1-\theta-q)}{x-\mathbf{x}^{(\infty)}}\right)^2+\left(\theta+\frac{(\theta-1)(1-\theta-q)}{x-\mathbf{x}^{(\infty)}}\right)+1$. 

Let us show for any $x \in \mathcal{A}_2,\ f_{\theta,q,3}(x) \in \mathcal{A}_0.$  Indeed, for any $x \in \mathcal{A}_2,$
$$
\left|\frac{(\theta-1)(1-\theta-q)}{x-\mathbf{x}^{(\infty)}}\right|_p=\frac{|q|_p|\theta-1|_p}{|x-\mathbf{x}^{(\infty)}|_p}<1
$$
which implies $|g_1(x)|_p= 1$ and
$$
\left|f_{\theta, q, 3}(x)-\mathbf{x}^{(0)}\right|_p=\frac{\left|\theta-1\right|_p}{\left|x-\mathbf{x}^{(\infty)}\right|_p}\left|x-\mathbf{x}^{(0)}\right|_p<\left|x-\mathbf{x}^{(0)}\right|_p=|q|_p.
$$
Therefore, we have shown $\mathcal{A}_0 \cup \mathcal{A}_1 \cup \mathcal{A}_2 \cup \mathcal{A}_{0,\infty}\subset f^{-1}_{\theta,q,3}\left(\mathcal{A}_0\right)$.

(ii) We want to show that $f_{\theta,q,3}(x) \in \mathcal{A}_1$ for any $x \in \mathcal{A}_{\infty} \cup \mathcal{A}_3\cup\mathcal{A}_{2,3,\infty}.$ 

Let $x \in \mathcal{A}_{\infty}$. We have 
$$
\left|\frac{(\theta-1)(1-\theta-q)}{x-\mathbf{x}^{(\infty)}}\right|_p=\frac{|q|_p|\theta-1|_p}{|x-\mathbf{x}^{(\infty)}|_p}>1
$$
which implies
$$
\left|f_{\theta,q,3}(x)\right|_p=\left|\theta+\frac{(\theta-1)(1-\theta-q)}{x-\mathbf{x}^{(\infty)}}\right|_p^3 > 1 > |q|_p.
$$
Then by \eqref{IsingPottsSingular}, $\left|f_{\theta,q,3}(x)-\mathbf{x}^{(0)}\right|_p=\left|f_{\theta,q,3}(x)\right|_p>|q|_p.$ Thus,  $f_{\theta,q,3}\left(\mathcal{A}_{\infty}\right)\subset\mathcal{A}_1$. 

Let $x \in \mathcal{A}_{3}$. We have  
$$
\left|\frac{(\theta-1)(1-\theta-q)}{x-\mathbf{x}^{(\infty)}}\right|_p=\frac{|\theta-1|_p|q|_p}{|x-\mathbf{x}^{(\infty)}|_p}<1
$$
which implies $\left|g_1(x)\right|_p= 1$. Then by \eqref{IsingPottsSingular2}, $$|q|_p<\left|f_{\theta,q,3}(x)-\mathbf{x}^{(0)}\right|_p=\frac{\left|\theta-1\right|_p\left|x-\mathbf{x}^{(0)}\right|_p}{\left|x-\mathbf{x}^{(\infty)}\right|_p}<1.$$ Hence, $f_{\theta,q,3}\left(\mathcal{A}_3\right)\subset\mathcal{A}_1$. 

Let $x \in \mathcal{A}_{2,3,\infty}.$ Set $y:=\frac{x-\mathbf{x}^{(\infty)}}{\theta -1}.$ Then, $|y|_p=|q|_p,\ y=\frac{y^{*}}{|q|_p}$ and
\begin{eqnarray*}
g(x)&=&(\theta x+q-1)^2+(\theta x+q-1)(x-\mathbf{x}^{(\infty)})+(x-\mathbf{x}^{(\infty)})^2\\
&=&(\theta-1)^2\left[(\theta^2+\theta+1)y^2+(2\theta+1)(1-\theta-q)y+(1-\theta-q)^2\right]\\
&=&\frac{(\theta-1)^2}{|q|^2_p}\left[(\theta^2+\theta+1)(y^{*})^2+(2\theta+1)(1-\theta-q)^{*}(y^{*})+\left((1-\theta-q)^{*}\right)^2\right].
\end{eqnarray*}
Since $\mathbf{x}^{(2)},\mathbf{x}^{(3)}\not\in \mathcal{A}_{2,3,\infty},$ then $|g(x)|_p=|q|^2_p|\theta-1|^2_p$ for any $x \in \mathcal{A}_{2,3,\infty}.$ It follows from \eqref{fminusx0} that $|f_{\theta,q,3}(x)-\mathbf{x}^{(0)}|_p=1>|q|_p$ for any $x\in \mathcal{A}_{2,3,\infty}$. This means $f_{\theta,q,3}\left(\mathcal{A}_{2,3,\infty}\right)\subset\mathcal{A}_1.$ By (i), $\mathcal{A}_1 \subset f^{-1}_{\theta,q,3}\left(\mathcal{A}_0\right)$, then the assertion follows.

(iii) Finally, let us show that $f_{\theta,q,3}(x) \in \mathcal{A}_{0,\infty}$ for any $x \in \mathcal{A}_{1,\infty}.$ 

We first show that $|f_{\theta,q,3}(x)-\mathbf{x}^{(0)}|_p=|q|_p$ for any  $x \in \mathcal{A}_{1,\infty}.$ Indeed, we have
\begin{eqnarray*}
	\left|\frac{(\theta-1)(1-\theta-q)}{x-\mathbf{x}^{(\infty)}}\right|_p=\frac{|\theta-1|_p|q|_p}{|x-\mathbf{x}^{(\infty)}|_p}<1
\end{eqnarray*} 
and $\left|g_1(x)-3\right|_p<\left|g_1(x)\right|_p= 1.$ By \eqref{IsingPottsSingular2},
\begin{eqnarray*}
\left|f_{\theta,q,3}(x)-\mathbf{x}^{(0)}\right|_p=\frac{\left|\theta-1\right|_p\left|x-\mathbf{x}^{(0)}\right|_p}{\left|x-\mathbf{x}^{(\infty)}\right|_p}=|q|_p.
\end{eqnarray*}

We now show that $|f_{\theta,q,3}(x)-\mathbf{x}^{(\infty)}|_p=|q|_p$ for any  $x \in \mathcal{A}_{1,\infty}.$ Since  $|f_{\theta,q,3}(x)-\mathbf{x}^{(\infty)}|_p=|f_{\theta,q,3}(x)-\mathbf{x}^{(0)}+q+(\theta-1)|_p=|f_{\theta,q,3}(x)-\mathbf{x}^{(0)}+q|_p,$ it suffices to show $|f_{\theta,q,3}(x)-\mathbf{x}^{(0)}+q|_p=|q|_p.$ In fact
\begin{multline*}
f_{\theta,q,3}(x)-\mathbf{x}^{(0)}+q=\frac{(\theta-1)(x-\mathbf{x}^{(0)})}{x-\mathbf{x}^{(\infty)}}\left(g_1(x)-3\right)\\+3\left((\theta-1)-\frac{(\theta-1)^2}{x-\mathbf{x}^{(\infty)}}\right)+\frac{q(\theta-1)(\mathbf{y}^{(1)}-3)}{x-\mathbf{x}^{(\infty)}}+q\frac{x-\mathbf{x}^{(1)}}{x-\mathbf{x}^{(\infty)}}.
\end{multline*}
Since $|\mathbf{y}^{(1)}-3|_p=|x-\mathbf{x}^{(0)}|_p=|q|_p,\ |x-\mathbf{x}^{(\infty)}|_p=|x-\mathbf{x}^{(1)}|_p=|\theta-1|_p, \ |g_1(x)-3|_p<1,$ we have $|f_{\theta,q,3}(x)-\mathbf{x}^{(0)}+q|_p=|q|_p.$ This means $|f_{\theta,q,3}(x)-\mathbf{x}^{(\infty)}|_p=|q|_p$.  Consequently, $f_{\theta,q,3}\left(\mathcal{A}_{1,\infty}\right)\subset\mathcal{A}_{0,\infty}$ or $\mathcal{A}_{1,\infty}\subset f^{-1}_{\theta,q,3}\left(\mathcal{A}_{0,\infty} \right)\subset f^{-2}_{\theta,q,3}\left(\mathcal{A}_0\right).$
\end{proof}

We now describe the attracting basin $\mathfrak{B}\left( \mathbf{x}^{(0)} \right)$.
\begin{theorem}\label{basin}
We have
$$\textup{\textbf{Dom}}\{f_{\theta,q,3}\}\setminus \left(\mathcal{C}_1 \cup \mathcal{C}_{2}\cup \mathcal{C}_{3}\right) \subset \mathfrak{B}\left( \mathbf{x}^{(0)} \right)=\bigcup\limits^{+\infty}_{n = 0}f^{-n}_{\theta, q, 3}\left(\mathcal{A}_0\right).$$ 
\end{theorem}

\begin{proof}
Let $x\in\mathcal{A}_0.$ It follows from Proposition \ref{p=1mod3} (i) that $$|f_{\theta, q, 3}(x)-\mathbf{x}^{(0)}|_p \leq \frac{|\theta-1|_p}{|q|_p}|x-\mathbf{x}^{(0)}|_p.$$ This implies that $f_{\theta, q, 3}\left(\mathbb{S}_r\left(\mathbf{x}^{(0)}\right)\right)\subset \mathbb{B}_{\frac{r|\theta-1|_p}{|q|_p}}\left(\mathbf{x}^{(0)}\right)$ where $\mathbb{S}_{r}\left(\mathbf{x}^{(0)}\right)$ is the sphere of radius $r$ centered at  $\mathbf{x}^{(0)}$. Observing $\frac{|\theta-1|_p}{|q|_p}<1$, we have in particular,  $f_{\theta, q, 3}\left(\mathcal{A}_0\right)\subset \mathcal{A}_0$. Let $x\in \mathfrak{B}\left(\mathbf{x}^{(0)}\right).$ Then there exist some $n_0\in\mathbb{N}$ and sufficiently small $r>0$ such that $f_{\theta, q, 3}^{n_0}(x)\in\mathbb{S}_r(\mathbf{x}^{(0)})\subset \mathcal{A}_0.$ This means that $x\in f_{\theta,q,3}^{-n_0}\left(\mathcal{A}_0\right)$ or equivalently
$$\mathfrak{B}\left(\mathbf{x}^{(0)} \right) = \bigcup^{+\infty}_{n = 0}f_{\theta,q,3}^{-n}\left(\mathcal{A}_0\right).$$
Moreover, due to Proposition  \ref{p=1mod3} (i)-(iii), we have
$$
\textup{\textbf{Dom}}\{f_{\theta,q,3}\}\setminus \left(\mathcal{C}_1 \cup \mathcal{C}_{2}\cup \mathcal{C}_{3}\right) \subset f_{\theta,q,3}^{-1}\left(\mathcal{A}_0\right)\cup f_{\theta,q,3}^{-2}\left(\mathcal{A}_0\right)\subset \mathfrak{B}\left(\mathbf{x}^{(0)} \right). 
$$
\end{proof}

\subsection{Chaotic behavior near the repelling fixed points} Let us consider the sets $\mathcal{C}_1, \mathcal{C}_{2}, \mathcal{C}_{3}$.

\begin{proposition}\label{ratiooffandx}
The following equality holds.
\begin{equation}\label{fxfx}
\left|f_{\theta, q, 3}(\bar{x}) -f_{\theta, q, 3}(\bar{\bar{x}})\right|_p=
\begin{cases}
\frac{|q|_p}{|\theta-1|_p}|\bar{x}-\bar{\bar{x}}|_p & \textup{if} \ \ \bar{x},\bar{\bar{x}}\in\mathcal{C}_1 \\
\frac{1}{|q|_p|\theta-1|_p}|\bar{x}-\bar{\bar{x}}|_p & \textup{if} \ \ \bar{x},\bar{\bar{x}}\in\mathcal{C}_2 \\
\frac{1}{|q|_p|\theta-1|_p}|\bar{x}-\bar{\bar{x}}|_p & \textup{if} \ \ \bar{x},\bar{\bar{x}}\in\mathcal{C}_3.
\end{cases}
\end{equation}
\end{proposition}

\begin{proof}
We have
\begin{equation}\label{distancexy}
f_{\theta, q, 3}(\bar{x}) -f_{\theta, q, 3}(\bar{\bar{x}})=
\frac{(\theta-1)(1-\theta-q)(\bar{\bar{x}}-\bar{x})}{(\bar{x}-\mathbf{x}^{\infty})(\bar{\bar{x}}-\mathbf{x}^{\infty})}F_{\theta, q, 3}\left(\bar{x},\bar{\bar{x}}\right)
\end{equation}
where
\begin{multline*}
F_{\theta, q, 3}(\bar{x},\bar{\bar{x}})=3\theta^2+3\theta(1-\theta-q)\left(\frac{\theta-1}{\bar{x}-\mathbf{x}^{\infty}}+\frac{\theta-1}{\bar{\bar{x}}-\mathbf{x}^{\infty}}\right),\\
+(1-\theta-q)^2\left(\frac{(\theta-1)^2}{\left(\bar{x}-\mathbf{x}^{\infty}\right)^2}+\frac{(\theta-1)^2}{\left(\bar{x}-\mathbf{x}^{\infty}\right)\left(\bar{\bar{x}}-\mathbf{x}^{\infty}\right)}+\frac{(\theta-1)^2}{\left(\bar{\bar{x}}-\mathbf{x}^{\infty}\right)^2}\right).
\end{multline*}
Let $\bar{x},\bar{\bar{x}}\in\mathcal{C}_1.$ Then
\begin{eqnarray*}
\left|3\theta(1-\theta-q)\left(\frac{\theta-1}{\bar{x}-\mathbf{x}^{\infty}}+\frac{\theta-1}{\bar{\bar{x}}-\mathbf{x}^{\infty}}\right)\right|_p\leq |q|_p,\\
\left|(1-\theta-q)^2\left(\frac{(\theta-1)^2}{\left(\bar{x}-\mathbf{x}^{\infty}\right)^2}+\frac{(\theta-1)^2}{\left(\bar{x}-\mathbf{x}^{\infty}\right)\left(\bar{\bar{x}}-\mathbf{x}^{\infty}\right)}+\frac{(\theta-1)^2}{\left(\bar{\bar{x}}-\mathbf{x}^{\infty}\right)^2}\right) \right|_p\leq|q|_p^2\\
\end{eqnarray*}
which imply that $|F_{\theta, q, 3}(\bar{x},\bar{\bar{x}})|_p=1$.
Thus, $\left|f_{\theta, q, 3}(\bar{x}) -f_{\theta, q, 3}(\bar{\bar{x}})\right|_p=\frac{|q|_p}{|\theta-1|_p}|\bar{x}-\bar{\bar{x}}|_p.$ Let $\bar{x},\bar{\bar{x}}\in\mathcal{C}_i$ for $i=2,3.$ Then
\begin{eqnarray*}
\frac{\mathbf{x}^{(i)}-\mathbf{x}^{\infty}}{\theta-1}=\mathbf{y}^{i}=\frac{y^{(i)}_0+y^{(i)}_1p+\dots}{|q|_p}, \quad i=2,3,
\end{eqnarray*}
where $y^{(2)}_0$ and $y^{(3)}_0$ are roots of the following congruent equation
\begin{equation*}
3t^2+3(1-\theta-q)^*t+\left((1-\theta-q)^*\right)^2 \equiv 0 \ (mod \ p).
\end{equation*}
Since $\bar{x},\bar{\bar{x}}\in\mathcal{C}_i$ for $i=2,3,$ we obtain
\begin{eqnarray*}
\left|\frac{\bar{x}-\mathbf{x}^{\infty}}{\theta-1}\right|_p=\left|\frac{\bar{\bar{x}}-\mathbf{x}^{\infty}}{\theta-1}\right|_p&=&|q|_p,\\
\left(\frac{\bar{x}-\mathbf{x}^{\infty}}{\theta-1}\right)^{*}=\left(\frac{\bar{\bar{x}}-\mathbf{x}^{\infty}}{\theta-1}\right)^{*}&\equiv& y^{(i)}_0 \ (mod \ p).
\end{eqnarray*}  
Then
\begin{multline*}
	F_{\theta, q, 3}(\bar{x},\bar{\bar{x}})=3\theta^2+3\theta(1-\theta-q)^{*}\left(\left(\frac{\theta-1}{\bar{x}-\mathbf{x}^{\infty}}\right)^{*}+\left(\frac{\theta-1}{\bar{\bar{x}}-\mathbf{x}^{\infty}}\right)^{*}\right)\\
	+\left((1-\theta-q)^{*}\right)^2\left(\left(\left(\frac{\theta-1}{\bar{x}-\mathbf{x}^{\infty}}\right)^{*}\right)^2+\left(\frac{\theta-1}{\bar{x}-\mathbf{x}^{\infty}}\right)^{*}\left(\frac{\theta-1}{\bar{\bar{x}}-\mathbf{x}^{\infty}}\right)^{*}+\left(\left(\frac{\theta-1}{\bar{x}-\mathbf{x}^{\infty}}\right)^{*}\right)^2\right),
\end{multline*}
and hence
\begin{eqnarray*}
\left(y^{(i)}_0\right)^2F_{\theta, q, 3}(\bar{x},\bar{\bar{x}})&\equiv& 3\left(y^{(i)}_0\right)^2+6(1-\theta-q)^{*}\left(y^{(i)}_0\right)+3\left((1-\theta-q)^{*}\right)^2\\
&\equiv& 3\left(\left(y^{(i)}_0\right)+(1-\theta-q)^{*}\right)^2\not\equiv 0 \ ({\rm mod} \ p).
\end{eqnarray*}
This implies $|F_{\theta, q, 3}(\bar{x},\bar{\bar{x}})|_p=1$. Therefore, $\left|f_{\theta, q, 3}(\bar{x}) -f_{\theta, q, 3}(\bar{\bar{x}})\right|_p=\frac{|\bar{x}-\bar{\bar{x}}|_p}{|q|_p|\theta-1|_p}$ for any $\bar{x},\bar{\bar{x}}\in\mathcal{C}_i, \ i=2,3.$
\end{proof}

To continue the study, we distinguish two cases, $0<|\theta-1|_p < |q|^2_p < 1$ and $0 < |q|^2_p\leq |\theta-1|_p < |q|_p < 1$. Throughout this part, we always suppose $$r=|q|_p|\theta-1|_p.$$
\subsubsection{Case $0<|\theta-1|_p < |q|^2_p < 1$}

For $i=1,2,3$, we consider disjoint open balls $\mathbb{B}_r\left(\mathbf{x}^{(i)}\right)$. It is clear that $\mathbb{B}_r\left(\mathbf{x}^{(1)}\right) \subsetneq \mathcal{C}_{1}$ and $\mathbb{B}_r\left(\mathbf{x}^{(i)}\right)=\mathcal{C}_i$ for $i=2,3$. We define the set $$\mathbb{X}=\mathbb{B}_r\left(\mathbf{x}^{(1)}\right)\cup\mathbb{B}_r\left(\mathbf{x}^{(2)}\right)\cup\mathbb{B}_r\left(\mathbf{x}^{(3)}\right).$$
Then $\left(\mathcal{C}_1\cup\mathcal{C}_2\cup\mathcal{C}_3\right)\setminus\mathbb{X}=\mathcal{C}_1\setminus\mathbb{B}_r\left(\mathbf{x}^{(1)}\right)$.

\begin{proposition}\label{propertyofX}
Let $0<|\theta-1|_p < |q|^2_p < 1.$ Then following inclusions hold:
\begin{itemize}
	\item[(i)] for all $i=1,2,3,\ $ $\mathbb{X}\subset \mathcal{C}_1\cup\mathcal{C}_2\cup\mathcal{C}_3\subset f_{\theta, q, 3}\left(\mathbb{B}_r\left(\mathbf{x}^{(i)}\right)\right)$;
	\item[(ii)] $f_{\theta, q, 3}\left(\mathcal{C}_1\setminus\mathbb{B}_r\left(\mathbf{x}^{(1)}\right)\right)\subset \mathbb{Q}_p\setminus\left(\mathcal{C}_1\cup\mathcal{C}_2\cup\mathcal{C}_3\right);$  
	\item[(iii)] $f^{-1}_{\theta, q, 3}\left(\mathbb{X}\right)\subset \mathbb{X}.$
\end{itemize}
\end{proposition}

\begin{proof}

(i) It is clear that $\mathbb{B}_r\left(\mathbf{x}^{(i)}\right)\subseteq \mathcal{C}_i$ for $i=1,2,3.$ So, $\mathbb{X} \subset \mathcal{C}_1\cup\mathcal{C}_2\cup\mathcal{C}_3$. We need only show  $\mathcal{C}_1\cup\mathcal{C}_2\cup\mathcal{C}_3\subset f_{\theta, q, 3}\left(\mathbb{B}_r\left(\mathbf{x}^{(i)}\right)\right)$ for $i=1,2,3.$

By \eqref{fxfx}, for any $x\in \mathbb{B}_r\left(\mathbf{x}^{(1)}\right)$ we have
$\left|f_{\theta, q, 3}(x)-\mathbf{x}^{(1)}\right|_p=\frac{|q|_p}{|\theta-1|_p}|x-\mathbf{x}^{(1)}|_p$. Thus $f_{\theta, q, 3}\left(\mathbb{B}_r\left(\mathbf{x}^{(1)}\right)\right)=\mathbb{B}_{|q|^2}\left(\mathbf{x}^{(1)}\right)$. Since $|\mathbf{x}^{(1)}-\mathbf{x}^{(2)}|_p=|\mathbf{x}^{(1)}-\mathbf{x}^{(3)}|_p=|\theta-1|_p<|q|^2_p,$ we have  $\mathbf{x}^{(1)}, \mathbf{x}^{(2)},$ $\mathbf{x}^{(3)}\in \mathbb{B}_{|q|^2_p}\left(\mathbf{x}^{(1)}\right).$ Since $|x-\mathbf{x}^{(1)}|_p\leq |\theta-1|_p<|q|_p^2$ for any $x\in\mathcal{C}_1\cup\mathcal{C}_2\cup\mathcal{C}_3,$ we obtain $\mathcal{C}_1\cup\mathcal{C}_2\cup\mathcal{C}_3\subset f_{\theta, q, 3}\left(\mathbb{B}_{r}\left(\mathbf{x}^{(1)}\right)\right).$

By \eqref{fxfx}, for $i=2,3$ and any $x\in \mathbb{B}_r\left(\mathbf{x}^{(i)}\right)$ we have $\left|f_{\theta, q, 3}(x)-\mathbf{x}^{(i)}\right|_p=\frac{|x-\mathbf{x}^{(i)}|_p}{|q|_p|\theta-1|_p}$. Then $f_{\theta, q, 3}\left(\mathbb{B}_r\left(\mathbf{x}^{(i)}\right)\right)=\mathbb{B}_{1}\left(\mathbf{x}^{(i)}\right), \ \ i=2,3$. Since $|\mathbf{x}^{(1)}-\mathbf{x}^{(2)}|_p=|\mathbf{x}^{(1)}-\mathbf{x}^{(3)}|_p=|\theta-1|_p<1,$ and $|\mathbf{x}^{(2)}-\mathbf{x}^{(3)}|_p=|q|_p|\theta-1|_p<1,$ we deduce that $\mathbf{x}^{(1)}, \mathbf{x}^{(2)}, \mathbf{x}^{(3)}\in \mathbb{B}_{1}\left(\mathbf{x}^{(i)}\right)$ for  $i=2,3.$ Since $|x-\mathbf{x}^{(i)}|_p\leq |\theta-1|_p<1$ for any $x\in\mathcal{C}_1\cup\mathcal{C}_2\cup\mathcal{C}_3,$ we conclude $\subset\mathcal{C}_1\cup\mathcal{C}_2\cup\mathcal{C}_3\subset f_{\theta, q, 3}\left(\mathbb{B}_r\left(\mathbf{x}^{(i)}\right)\right)$ for  $i=2,3.$

(ii) Let $x\in \mathcal{C}_1\setminus\mathbb{B}_r\left(\mathbf{x}^{(1)}\right)$. Note $$\mathcal{C}_1\setminus\mathbb{B}_r\left(\mathbf{x}^{(1)}\right)=\left\{ x \in \mathbb{Q}_p : |q|_p|\theta-1|_p\leq\left|x-\mathbf{x}^{(1)}\right|_p < \left|x-\mathbf{x}^{(\infty)}\right|_p = |\theta-1|_p \right\}.$$ By \eqref{fxfx} and the fact $|\mathbf{x}^{(1)}-\mathbf{x}^{(2)}|_p=|\mathbf{x}^{(1)}-\mathbf{x}^{(3)}|_p=|\theta-1|_p$, we have
$$
|q|_p|\theta-1|_p<|\theta-1|_p<|q|_p^2\leq\left|f_{\theta, q, 3}(x)-\mathbf{x}^{(1)}\right|_p=\frac{|q|_p|x-\mathbf{x}^{(1)}|_p}{|\theta-1|_p}<|q|_p
$$ and
$$
|q|_p^2\leq\left|f_{\theta, q, 3}(x)-\mathbf{x}^{(i)}\right|_p=\left|f_{\theta, q, 3}(x)-\mathbf{x}^{(1)}+\mathbf{x}^{(1)}-\mathbf{x}^{(i)}\right|_p=\left|f_{\theta, q, 3}(x)-\mathbf{x}^{(1)}\right|_p<|q|_p
$$
for $i=2,3$. This means that $f_{\theta, q, 3}(x)\not\in \mathcal{C}_1\cup\mathcal{C}_2\cup\mathcal{C}_3.$

(iii) The inclusion $f^{-1}_{\theta, q, 3}\left(\mathbb{X}\right)\subset \mathbb{X}$ follows from Proposition \ref{p=1mod3} and (ii) which $$f_{\theta,q,3}\left(\textup{\textbf{Dom}}\{f_{\theta,q,3}\}\setminus \mathbb{X}\right) \subset \mathbb{Q}_p\setminus \mathbb{X}.$$
\end{proof}

We define the Julia set of the Potts--Bethe mapping as follows
\begin{equation*}
\mathcal{J}_0 = \bigcap_{n=0}^{\infty}f^{-n}_{\theta, q, 3}\left(\mathbb{X}\right).
\end{equation*}

\begin{theorem}\label{decompositionofQp}
Let $0<|\theta-1|_p < |q|^2_p < 1.$ Then,
\begin{itemize}
\item[(i)] the dynamical system $f_{\theta, q, 3}:\mathcal{J}_0\to\mathcal{J}_0$ is isometrically conjugate to the full shift $(\Sigma_3,\sigma,d_{f_{\theta, q, 3}})$ over an alphabet of three symbols;
\item[(ii)] one has $$\mathbb{Q}_p=\mathfrak{B}\left( \mathbf{x}^{(0)} \right)\cup \mathcal{J}_0\cup \bigcup\limits_{n=0}^{+\infty}f_{\theta, q, 3}^{-n}\{\mathbf{x}^{(\infty)}\}.$$
\end{itemize}
\end{theorem}

\begin{proof}

(i) By \eqref{fxfx}, Proposition \ref{propertyofX} and Theorem \ref{conjugacytheorem}, the dynamical system $f_{\theta, q, 3}:\mathcal{J}_0\to\mathcal{J}_0$ is isometrically conjugate to the full shift dynamics $(\Sigma_3,\sigma,d_{f_{\theta, q, 3}})$.

(ii) Due to Theorem \ref{basin}, the following inclusion holds
$$\textup{\textbf{Dom}}\{f_{\theta,q,3}\}\setminus \left(\mathcal{C}_1 \cup \mathcal{C}_{2}\cup \mathcal{C}_{3}\right) \subset \mathfrak{B}\left( \mathbf{x}^{(0)} \right).$$ 

Let $x\in \left(\mathcal{C}_1 \cup \mathcal{C}_{2}\cup \mathcal{C}_{3}\right)\setminus \mathcal{J}_0.$ There exists $n_0$ such that $x\not\in f^{-n_0}_{\theta, q, 3}\left(\mathbb{X}\right)$ or equivalently $f^{n_0}_{\theta, q, 3}(x)\not\in\mathbb{X}.$ Due to Proposition \ref{propertyofX} (iii), we get  $f^{n_0+1}_{\theta, q, 3}(x)\in \mathfrak{B}\left(\mathbf{x}^{(0)}\right)\cup\{\mathbf{x}^{\infty}\}.$ Consequently, we obtain the following decomposition $$\mathbb{Q}_p=\mathfrak{B}\left( \mathbf{x}^{(0)} \right)\cup \mathcal{J}_0\cup \bigcup\limits_{n=0}^{+\infty}f_{\theta, q, 3}^{-n}\{\mathbf{x}^{(\infty)}\}.$$
\end{proof}

\subsubsection{Case $0<|q|_p^2\leq|\theta-1|_p<|q|_p<1$}

Let $m\geq 1$ be the integer such that $$|q|_p^\frac{m+1}{m} \leq |\theta-1|_p < |q|_p^\frac{m+2}{m+1}.$$ We consider for $l=0,1,\dots,m$, the sphere $$\mathcal{G}_l :=\left\{x : \left|x-\mathbf{x}^{(1)} \right| = \frac{|\theta-1|_p^{l+1}}{|q|_p^{l}}\right\}.$$ Recall $r=|q|_p|\theta-1|_p.$ Then $\mathcal{G}_l$ can be written as a union of disjoint balls of radius $r$, i.e., $\mathcal{G}_l = \bigcup\limits_{i=1}^{m_l}\mathbb{B}_{r}(a_{i,l})$ where $a_{i,l} \in \mathcal{G}_l$ and $\mathbb{B}_{r}(a_{i,l}) \cap \mathbb{B}_{r}(a_{j,l}) = \emptyset$ for $1 \leq i \neq j \leq m_l$. Since the sphere $\mathcal{G}_l$ is a union of $p-1$ open balls with radius $\frac{|\theta-1|_p^{l+1}}{|q|_p^{l}},$
 by the following Remark \ref{disjointballs}, $m_l=(p-1)\frac{|\theta-1|_p^l}{|q|_p^{l+1}}$.

\begin{remark}\label{disjointballs}
Let $\gamma\leq \delta$ be two integers. Any open ball $\mathbb{B}_{p^\delta}(a)$ is decomposed into $p^{\delta-\gamma}$ disjoint open balls $\mathbb{B}_{p^\gamma}(a_i)$, i.e., $\mathbb{B}_{p^\delta}(a) = \bigcup\limits_{i=1}^{p^{\delta-\gamma}}\mathbb{B}_{p^\gamma}(a_i)$ where $a_i \in \mathbb{B}_{p^\delta}(a)$, and $|a_i-a_j|_p=p^{\gamma}$ if $a_i\neq a_j$.
\end{remark}

\begin{proposition}\label{1-2-1}
The map $f_{\theta, q, 3}$ is a one-to-one between $\mathcal{C}_1$ and $\mathbb{B}_{|q|_p}\left(\mathbf{x}^{(1)}\right)$. In particular, it is a one-to-one between $\mathcal{G}_l$ and $\mathcal{G}_{l-1}$ for $1\leq l\leq m.$
\end{proposition}

\begin{proof}
Recall $\mathcal{C}_1=\mathbb{B}_{|\theta-1|_p}\left(\mathbf{x}^{(1)}\right).$ By \eqref{fxfx}, we have scaling on $\mathcal{C}_1$ such that for any $x,y\in\mathcal{C}_1$ \begin{eqnarray*}
\left|f_{\theta, q, 3}(x)-f_{\theta, q, 3}(y)\right|_p&=&\frac{|q|_p}{|\theta-1|_p}|x-y|_p.
\end{eqnarray*}

Thus, by Lemma \ref{scaling}, $f_{\theta,q,3}$ is a one-to-one between $\mathcal{C}_1$ and $\mathbb{B}_{|q|_p}\left(\mathbf{x}^{(1)}\right)$. We have for $1\leq l\leq m,\ \mathcal{G}_l \subset \mathcal{C}_1$ and $f_{\theta,q,3}\left(\mathcal{G}_l\right) \subseteq \mathcal{G}_{l-1}$. Since $\mathcal{G}_l$ are spheres, by Corollary \ref{scaling2}, $f_{\theta,q,3}$ is one-to-one between $\mathcal{G}_l$ and $\mathcal{G}_{l-1}.$
\end{proof}

It is clear that for all $l=1,2,\dots,m,\ \mathcal{G}_l \subset \mathcal{C}_1$. For all $l=1,2,\dots,m,$ we let  $\mathcal{H}_l = \bigcup\limits_{i=1}^{n_l}\mathbb{B}_{r}(a_{i,l}) \subset \mathcal{G}_l$ with $a_{i,l} \in \mathcal{G}_l$ such that for $1 \leq i \neq j \leq n_l,\ \mathbb{B}_{r}(a_{i,l}) \cap \mathbb{B}_{r}(a_{j,l}) = \emptyset$ and $\left(\mathcal{C}_2 \cup \mathcal{C}_3 \right) \subset f_{\theta, q, 3}^{l}\left(\mathbb{B}_{r}(a_{i,l})\right)$.

\begin{proposition}\label{numberballs}
For all $l=1,2,\dots,m$, we have $n_l=1.$
\end{proposition}

\begin{proof}
We need only to show $n_1=1$ and the rest follows from Proposition \ref{1-2-1}. 
By the fact $\left|\mathbf{x}^{(2)}-\mathbf{x}^{(3)}\right|=|q|_p|\theta-1|_p$, we have $|x-y|=|q|_p|\theta-1|_p$ for any $x \in \mathcal{C}_2,\ y \in \mathcal{C}_3$. Since $|q|_p|\theta-1|_p < |q|_p^2$,  $\mathcal{C}_2$ and $\mathcal{C}_3$ are in the same image of some ball $\mathbb{B}_r(a_{i,1}) \subset \mathcal{G}_1$ with $1 \leq i \leq m_l$, i.e., $\mathcal{C}_2 \cup \mathcal{C}_3 \subset f_{\theta,q,3}\left(\mathbb{B}_r(a_{i,1})\right)$. By Proposition \ref{1-2-1}, this ball is unique.
\end{proof}

Without loss of generality, we assume $\mathcal{H}_l = \mathbb{B}_r(a_{1,l})$ with $a_{1,l} \in \mathcal{G}_l$.

Let $\mathcal{D}_1=\mathbb{B}_{r}(\mathbf{x}^{(1)})$ and $\mathbb{X}_m= \mathcal{D}_1 \cup \mathcal{H}_m \cup \cdots \cup \mathcal{H}_1 \cup \mathcal{C}_2 \cup \mathcal{C}_3$ for $m\geq 1$.

\begin{proposition} \label{5.4}
Suppose $0 < |q|_p^\frac{m+1}{m} \leq |\theta-1|_p < |q|_p^\frac{m+2}{m+1} < 1$ for some $m \geq 1$. Then $f_{\theta, q, 3}^{-1}\left(\mathbb{X}_m\right) \subset \mathbb{X}_m.$
\end{proposition}

\begin{proof}

It suffices to show $f_{\theta, q, 3}\left(\textup{\textbf{Dom}}\{f_{\theta,q,3}\} \setminus \mathbb{X}_m\right) \cap \mathbb{X}_m = \emptyset.$ Suppose $\mathcal{C}=\mathcal{C}_1 \cup \mathcal{C}_{2}\cup \mathcal{C}_{3}.$ By Proposition \ref{p=1mod3}, we have $f_{\theta, q, 3}\left(\textup{\textbf{Dom}}\{f_{\theta,q,3}\} \setminus \mathcal{C}\right) \cap \mathcal{C} = \emptyset$. Since $\mathbb{X}_m \subset \mathcal{C},$ we are left to show $f_{\theta, q, 3}(\mathcal{C} \setminus \mathbb{X}_m) \cap \mathbb{X}_m = \emptyset.$ One can see $\mathcal{C} \setminus \mathbb{X}_m=\mathcal{C}_1 \setminus \mathbb{X}_m = \mathcal{C}_1 \setminus \left(\mathcal{D}_1 \cup \mathcal{H}_1 \cup \cdots \cup \mathcal{H}_m\right).$

By the one-to-one of $f_{\theta,q,3}$ in Proposition \ref{1-2-1}, uniqueness in \ref{numberballs} and definition of $\mathcal{H}_l,\ 1\leq l\leq m$, we have $f_{\theta,q,3}\left(\mathcal{C}_1 \setminus (\mathcal{D}_1 \cup \mathcal{H}_1 \cup \cdots \cup \mathcal{H}_m\right)) \cap \mathbb{X}_m = \emptyset.$
\end{proof}

\begin{proposition} \label{5.6}
Suppose $0 < |q|_p^\frac{m+1}{m} \leq |\theta-1|_p < |q|_p^\frac{m+2}{m+1}$ for some $m\geq 1$. Then
\begin{itemize}
\item[(I)] $f_{\theta, q, 3}(\mathcal{C}_2) \supset \left(\mathcal{C}_1 \cup \mathcal{C}_2 \cup \mathcal{C}_3\right)$ and $f_{\theta, q, 3}(\mathcal{C}_3) \supset \left(\mathcal{C}_1 \cup \mathcal{C}_2 \cup \mathcal{C}_3\right);$
\item[(II)] $f_{\theta, q, 3}(\mathcal{D}_1) \cap (\mathcal{C}_2 \cup \mathcal{C}_3) = \emptyset;$
\item[(III)] $f_{\theta, q, 3}\left(\mathcal{H}_1\right)\supset \left(\mathcal{C}_2 \cup  \mathcal{C}_3\right);$
\item[(IV)] for $m=1,$ we have $f_{\theta, q, 3}(\mathcal{D}_1) \supset \mathcal{H}_1$ and $f_{\theta, q, 3}\left(\mathcal{H}_1\right) \cap \mathcal{D}_1 = \emptyset.$
\item[(V)] for $m\geq 2,$ we have
\begin{itemize}
\item[(i)] $f_{\theta, q, 3}(\mathcal{D}_1) \supset \mathcal{H}_m$ and $f_{\theta, q, 3}(\mathcal{D}_1) \cap (\mathcal{H}_{m-1} \cup \cdots \cup \mathcal{H}_1) = \emptyset;$ 
\item[(ii)] for any $\mathcal{H}_l,\ 2\leq l\leq m,$

$f_{\theta, q, 3}\left(\mathcal{H}_l\right)\supset \mathcal{H}_{l-1}$ and $f_{\theta, q, 3}\left(\mathcal{H}_l\right) \cap \left(\mathcal{D}_1 \cup \mathcal{C}_2 \cup \mathcal{C}_3\bigcup\limits_{i=1, i\neq l-1}^{m}\mathcal{H}_i\right) = \emptyset;$
\item[(iii)] $f_{\theta, q, 3}\left(\mathcal{H}_1\right) \cap \left(\mathcal{D}_1 \cup \bigcup\limits_{l=2}^{m}\mathcal{H}_l\right) = \emptyset.$
\end{itemize}
\end{itemize}
\end{proposition}

\begin{proof}
Recall for $1\leq l\leq m,\ \mathcal{H}_l = \mathbb{B}_{r}(a_{1,l}),\ |a_{1,l}-\mathbf{x}^{(1)}|_p=\frac{|\theta-1|_p^{l+1}}{|q|_p^{l}}$ and $\left(\mathcal{C}_2 \cup \mathcal{C}_3 \right) \subset f_{\theta, q, 3}^{l}\left(\mathbb{B}_{r}(a_{1,l})\right).$

(I) By \eqref{fxfx}, we have for $i=2,3,$ $f_{\theta, q, 3}\left(\mathcal{C}_i\right) = \mathbb{B}_{1}\left(\mathbf{x}^{(i)}\right).$ Since $|\mathbf{x}^{(1)}-\mathbf{x}^{(2)}|_p=|\mathbf{x}^{(1)}-\mathbf{x}^{(3)}|_p=|\theta-1|_p<1$ and $|\mathbf{x}^{(2)}-\mathbf{x}^{(3)}|_p=|q|_p|\theta-1|_p<1,$ we have $\mathbf{x}^{(1)}, \mathbf{x}^{(2)}, \mathbf{x}^{(3)}\in \mathbb{B}_{1}\left(\mathbf{x}^{(i)}\right)$ for  $i=2,3.$ Moreover, since $|x-\mathbf{x}^{(i)}|_p\leq |\theta-1|_p<1$ for any $x\in\mathcal{C}_1\cup\mathcal{C}_2\cup\mathcal{C}_3,$ we have $\mathcal{C}_1\cup\mathcal{C}_2\cup\mathcal{C}_3\subset f_{\theta, q, 3}\left(\mathcal{C}_i\right).$

(II) By \eqref{fxfx}, we have $f_{\theta, q, 3}\left(\mathcal{D}_1\right) = \mathbb{B}_{|q|_p^2}\left(\mathbf{x}^{(1)}\right)$. Let $x \in \mathcal{C}_i,\ i=2,3.$ Then $|x-\mathbf{x}^{(1)}|_p=|x-\mathbf{x}^{(i)}+\mathbf{x}^{(i)}-\mathbf{x}^{(1)}|_p=|\mathbf{x}^{(i)}-\mathbf{x}^{(1)}|_p=|\theta-1|_p\geq|q|_p^2.$ Thus, $f_{\theta, q, 3}(\mathcal{D}_1) \cap (\mathcal{C}_2 \cup \mathcal{C}_3) = \emptyset$.

(III) By definition of $\mathcal{H}_1$, it is clear $\left(\mathcal{C}_2 \cup  \mathcal{C}_3\right) \subset f_{\theta, q, 3}\left(\mathcal{H}_1\right).$

(IV) Let now $m=1.$ Continue from the argument in (II). From $|q|_p^2 \leq |\theta-1|_p < |q|_p^{\frac{3}{2}}$, we get $\frac{|\theta-1|_p^2}{|q|_p}<|q|_p^2 \leq {|\theta-1|_p}$. Let $x \in \mathcal{H}_1$. Then $|x-\mathbf{x}^{(1)}|_p=|x-a_{1,1}+a_{1,1}-\mathbf{x}^{(1)}|_p=|a_{1,1}-\mathbf{x}^{(1)}|_p=\frac{|\theta-1|_p^{2}}{|q|_p} < |q|_p^2.$ Thus $f_{\theta, q, 3}(\mathcal{D}_1) \supset \mathcal{H}_1$.

By \eqref{fxfx}, we have $f_{\theta, q, 3}\left(\mathcal{H}_1\right) = \mathbb{B}_{|q|_p^2}(f_{\theta, q, 3}(a_{1,1}))$ and $|f_{\theta, q, 3}(a_{1,1})-\mathbf{x}^{(1)}|_p=|\theta-1|_p$. Let $x \in \mathcal{D}_1$. Then $|x-f_{\theta, q, 3}(a_{1,1})|_p=|x-\mathbf{x}^{(1)}+\mathbf{x}^{(1)}-f_{\theta, q, 3}(a_{1,1})|_p=|f_{\theta, q, 3}(a_{1,1})-\mathbf{x}^{(1)}|_p=|\theta-1|_p \geq |q|_p^2.$ Hence $f_{\theta, q, 3}(\mathcal{H}_1) \cap \mathcal{D}_1 = \emptyset$.

(V) Next, let $m\geq2.$ Continue from the argument in (II). From $0 < |q|_p^\frac{m+1}{m} \leq |\theta-1|_p < |q|_p^\frac{m+2}{m+1}$, we get $\frac{|\theta-1|_p^{m+1}}{|q|_p^{m}}<|q|_p^2 \leq \frac{|\theta-1|_p^{m}}{|q|_p^{m-1}}$. Let $x \in \mathcal{H}_m$. Then $|x-\mathbf{x}^{(1)}|_p=|x-a_{1,m}+a_{1,m}-\mathbf{x}^{(1)}|_p=|a_{1,m}-\mathbf{x}^{(1)}|_p=\frac{|\theta-1|_p^{m+1}}{|q|_p^m} < |q|_p^2.$ Thus $f_{\theta, q, 3}(\mathcal{D}_1) \supset \mathcal{H}_m$.

By Proposition \ref{1-2-1}, we have $f_{\theta,q,3}: \mathcal{C}_1 \to \mathbb{B}_{|q|_p}(\mathbf{x}^{(1)})$ and $f_{\theta,q,3}: \mathcal{G}_l \to \mathcal{G}_{l-1}$ for $1\leq l\leq m$ are one-to-one. By the definition of $\mathcal{H}_l$, we have $\mathcal{H}_{l-1} \subset f_{\theta,q,3}(\mathcal{H}_1 )= \mathbb{B}_{\bar{r}}(f_{\theta,q,3}(a_{1,l})),\ 2\leq l \leq m$ where $\bar{r} = |q|_p^2.$ Therefore, the rests, (i)--(iii), are true.
\end{proof}

We define the Julia set of the Potts--Bethe mapping as follows
\begin{equation*}
\mathcal{J}_m=\bigcap_{n=0}^{\infty}f^{-n}_{\theta, q, 3}\left(\mathbb{X}_m\right).
\end{equation*}

\begin{theorem} \label{5.7}
Let $0 < |q|_p^\frac{m+1}{m} \leq |\theta-1|_p < |q|_p^\frac{m+2}{m+1} < 1$ for some $m\geq 1$. Then
\begin{itemize}
\item[(i)] $(\mathbb{X}_m, \mathcal{J}_m, f_{\theta, q, 3})$ is transitive and the dynamical system $f_{\theta,q,3}:\mathcal{J}_m\to\mathcal{J}_m$ is isometrically conjugate to the subshift dynamics $(\Sigma_{A_m},\sigma_{A_m},d_{f_{\theta,q,3}})$ on $(m+3)$ symbols with
$$
A_m=
\begin{pmatrix}
\mathbf{e}_1^T & \mathbf{I}_{m} & \mathbf{0}^T & \mathbf{0}^T \\
0 & \mathbf{0} & 1 & 1 \\
1 & \mathds{1} & 1 & 1 \\
1 & \mathds{1} & 1 & 1
\end{pmatrix}
$$
where $\mathbf{I}_{m}$ is an $m \times m$ identity matrix, $\mathbf{e}_1=(1,0,\cdots,0)$, $\mathds{1}=(1,1,\cdots,1)$ and $\mathbf{0}=(0,0,\cdots,0)$. Here $\mathbf{v}^T$ is the transpose of vector $\mathbf{v}$. Moreover, this subshift is topologically conjugate to the full shift on three symbols.
\item[(ii)] $$\mathbb{Q}_p=\mathfrak{B}\left( \mathbf{x}^{(0)} \right)\cup \mathcal{J}_m\cup \bigcup\limits_{n=0}^{+\infty}f_{\theta, q, 3}^{-n}\{\mathbf{x}^{(\infty)}\}.$$
\end{itemize}
\end{theorem}

\begin{proof}
Let us consider the coding, $\mathcal{D}_1 \to 0,\ \mathcal{H}_l \to l$ where $l=1,\cdots,m,\ \mathcal{C}_2 \to m+1$ and $\mathcal{C}_3 \to m+2$. Due to Theorem \ref{conjugacytheorem} and Propositions \ref{numberballs}--\ref{5.6}, the dynamical system $f_{\theta, q, 3}:\mathcal{J}_m\to\mathcal{J}_m$ is isometrically conjugate to the subshift dynamics $(\Sigma_{A_m},\sigma_{A_m},d_{f_{\theta,q,3}})$ associated with the $(m+3)\times(m+3)$ matrix $A_m$ for $m\geq 1$.

Next, we show that the subshift $\Sigma_{A_m} \subset \Sigma_{m+3}$ is topologically conjugate to the full shift $\Sigma_3$ on three symbols. For details, one can refer to Sections 1.4--1.5 of \cite{LM}. We want to construct a bijection, $\pi : \Sigma_3 \to \Sigma_{A_m}$. First, we let $\pi_m : \mathcal{B}_{m+1}\left(\Sigma_A\right) \to \mathcal{B}_1\left(\Sigma_{A_m}\right),$ where $\mathcal{B}_N\left(\Sigma\right)$ is set of all $N$-block on any shift space $\Sigma$ over $\{1,2,\dots,s\}$ i.e., the $N$-block $\left(x_jx_{j+i}\cdots x_{j+N}\right),\ j\leq n\leq j+N,\ x_n \in \{1,2,\dots,s\}$ and $j\geq 1$. We define for $x_i \in \{0,1,2\}$ and $l,i=1,\dots,m,\ \pi_{m} :$
$$
\begin{matrix}
0^{m+1} \to 0 \\
0^{l}1x_1\cdots x_{m-l} \to l \\
0^{l}2x_1\cdots x_{m-l} \to l \\
1x_1\cdots x_m \to m+1 \\
2x_1\cdots x_m \to m+2.
\end{matrix}
$$

Now, we are ready to construct the map $\pi$. We have for any $x \in \Sigma_3$,
\begin{eqnarray}
\pi(x):=\pi_{m}(x_0\dots x_m)\pi_{m}(x_1\dots x_{m+1})\pi_{m}(x_2\dots x_{m+2})\dots\pi_{m}(x_i\dots x_{m+i})\dots. \nonumber
\end{eqnarray}
Then $\pi \circ \sigma = \sigma_{A_m} \circ \pi$. Moreover, $\pi$ is bijective. Therefore, we have shown a conjugacy between $\Sigma_A$ and $\Sigma_{A_m}$.

(ii) Due to Theorem \ref{basin}, the following inclusion holds
$$\textup{\textbf{Dom}}\{f_{\theta,q,3}\}\setminus \left(\mathcal{C}_1 \cup \mathcal{C}_{2}\cup \mathcal{C}_{3}\right) \subset \mathfrak{B}\left( \mathbf{x}^{(0)} \right).$$ 

Let $x\in \left(\mathcal{C}_1 \cup \mathcal{C}_{2}\cup \mathcal{C}_{3}\right)\setminus \mathcal{J}_m.$ There exists $n_0$ such that $x\not\in f^{-n_0}_{\theta, q, 3}\left(\mathbb{X}_m\right)$ or equivalently $f^{n_0}_{\theta, q, 3}(x)\not\in\mathbb{X}_m.$ Due to Proposition \ref{5.4}, we get  $f^{n_0+1}_{\theta, q, 3}(x)\in \mathfrak{B}\left(\mathbf{x}^{(0)}\right)\cup\{\mathbf{x}^{\infty}\}.$ Consequently, we get the following decomposition 
$$\mathbb{Q}_p=\mathfrak{B}\left( \mathbf{x}^{(0)} \right)\cup \mathcal{J}_m\cup \bigcup\limits_{n=0}^{+\infty}f_{\theta, q, 3}^{-n}\{\mathbf{x}^{(\infty)}\}.$$
\end{proof}

\section{Dynamics of the Potts--Bethe mapping for \texorpdfstring{$k=3$}{Lg} with small primes}

In this section, we investigate the dynamics of the Potts--Bethe mapping for $k=3$ in the cases $p=3$ and $p=2.$ The following techniques of Newton polygon and Hensel's lemma are very useful for finding the roots of polynomials in $\mathbb{Q}_p$.

\begin{definition}[Newton polygon] Let $f(x) = \sum_{i=1}^{n} a_ix^i$ be a polynomial whose coefficients are $p$-adic complex numbers. The Newton polygon of $f$ is the convex hull of the set of points $\left\{(i,-\log_p|a_i|_p) : 0\leq i\leq n\right\}.$ By convention, we set $-\log_p|0|_p = \infty$. To construct Newton polygon, we take a vertical ray starting at the point $(0,-\log_p|a_0|_p)$ and aiming down the $y$-axis. Then rotate the ray counterclockwise, keeping the point $(0,-\log_p|a_0|_p)$ fixed, until it bends around all of the points $(i,-\log_p|a_i|_p)$. \end{definition}

The following lemma is one of the results on Newton polygon.

\begin{lemma}[see Theorem 5.11 \cite{Silverman}] \label{newton}
Let $f(x) = \sum_{i=1}^{n} a_ix^i$ be a polynomial whose coefficients are $p$-adic complex numbers. Suppose that the Newton polygon of $f$ includes a line segment of slope $m$ whose horizontal length is $L$, i.e., the Newton polygon has a line segment from $$(i,-\log_p|a_i|_p) \ \text{to} \ (i+L,-\log_p|a_{i+L}|_p)$$ whose slope is $$m=\frac{-\log_p|a_{i+L}|_p+\log_p|a_i|_p}{L}.$$ Then $f$ has exactly $L$ roots $x_0$ (counting multiplicity), satisfying $|x_0|_p = p^{m}$. 
\end{lemma}

Next is the statement of Hensel's lemma.

\begin{lemma}[Hensel's lemma, see Theorem 3 of \cite{Bor Shaf}]\label{Hensel} \label{hensel}
Let $f$ be a polynomial whose coefficients are $p$-adic
integers. Let $\theta$ be a $p$-adic integer such that for some
$i\geq 0$ $$
f(\theta)\equiv 0 \ (mod \ p^{2i+1}), \quad
f'(\theta)\equiv 0 \ (mod \ p^{i}), \quad f'(\theta)\not\equiv 0 \ (mod \ p^{i+1}).
$$ Then $f$ has a unique $p$-adic integer root $x_0$ which satisfies $x_0\equiv \theta\ (mod \ p^{i+1}).$
\end{lemma}

Now, we begin our investigation by distinguishing two cases: $p=3$ and $p=2.$

\subsection{Case \texorpdfstring{$p=3$}{Lg}}
\subsubsection{Fixed points}

Let us find all possible roots of the cubic equation \eqref{wrty}. We always suppose $0<|\theta-1|_3<|q|_3<1.$

\begin{proposition}\label{rootp=3}
If $|q|_3=\frac{1}{3}$ then the cubic equation \eqref{wrty} has no roots over $\mathbb{Q}_3$. If $|q|_3<\frac{1}{3}$ then the cubic equation \eqref{wrty} has one root $\mathbf{y}^{(1)}$ over $\mathbb{Q}_3$ such that
\begin{itemize}
	\item[(i)] $|\mathbf{y}^{(1)}|_3 = 1;$
	\item[(ii)] $|\mathbf{y}^{(1)}-3|_3 < \frac{1}{3}.$
\end{itemize}
\end{proposition}

\begin{proof}
Let $|q|_3=3^{-t},\ |\theta-1|_3=3^{-s}.$ Then $s>t>0$. By assumptions, $q = 3^tq^{*}$ and $\theta = 1 + 3^s\theta^*$ with $|q^{*}|_3 = |\theta^*|_3=1$. In \eqref{wrty}, let $a_0=-(1-\theta-q)^2,\ a_1=-(2\theta+1)(1-\theta-q),\ a_2=-(1+\theta+\theta^2)$ and $a_3=1$. Then $|a_0|_3=|q|_3^2,\ |a_1|_3=\frac{|q|_3}{3},\ |a_2|_3=\frac{1}{3},\ |a_3|_3=1$ and $-\log_3|a_0|_3=2t,\ -\log_3|a_1|_3=t+1,\ -\log_3|a_2|_3=1,\ -\log_3|a_3|_3=0$. We consider the following points on the plane, $(0,2t),\ (1,t+1),\ (2,1),\ (3,0)$ and construct the Newton polygon. We distinguish two cases: $|q|_3=\frac{1}{3}$ and $|q|_3<\frac{1}{3}$.

\textsc{case 1:} $|q|_3=\frac{1}{3}$. Then $t=1$. The Newton polygon is the line from $(0,2)$ to $(3,0)$ with slope $m_1=-\frac{2}{3}$ and length $l_1=3$. So we have three roots $\mathbf{y}^{(1)},\ \mathbf{y}^{(2)},\ \mathbf{y}^{(3)} \in \mathbb{C}_3$ (the space of $3$-adic complex numbers) with $|\mathbf{y}^{(i)}|_3=3^{-\frac{2}{3}}$ for $i=1,2,3$ (including multiplicity). Since $-\frac{2}{3} \notin \mathbb{Z}$, the roots are not in $\mathbb{Q}_3$. Thus, the cubic equation \eqref{wrty} has no roots in $\mathbb{Q}_3$.

\textsc{case 2:} $|q|_3<\frac{1}{3}$. Then $t \geq 2$. The Newton polygon is the line from $(0,2t)$ to $(2,1)$ and from $(2,1)$ to $(3,0)$ with the slopes $m_1=-1$ and length $l_1=1$ and $m_2=\frac{1-2t}{2}$ and length $l_2=2$ respectively. For this case, we have three roots $\mathbf{y}^{(1)},\ \mathbf{y}^{(2)},\ \mathbf{y}^{(3)} \in \mathbb{C}_3$ with $|\mathbf{y}^{(1)}|_3=\frac{1}{3}$ and $|\mathbf{y}^{(i)}|_3=3^{\frac{1-2t}{2}}$ for $i=2,3$ (including multiplicity). Since $\frac{1-2t}{2} \notin \mathbb{Z}$, the roots $\mathbf{y}^{(2)},\ \mathbf{y}^{(3)}$ are not in $\mathbb{Q}_3$. Therefore, the equation \eqref{wrty} has at most one root in $\mathbb{Q}_3$.

Now, we show $\mathbf{y}^{(1)} \in \mathbb{Q}_3$. We consider the following cubic polynomial and its derivative,
\begin{eqnarray*}
h(y)&=&y^3-(\theta^2+\theta+1)y^2-(2\theta+1)(1-\theta-q)y-(1-\theta-q)^2, \\
h^{\prime}(y)&=&3y^2-2(\theta^2+\theta+1)y-(2\theta+1)(1-\theta-q).
\end{eqnarray*}
We have
\begin{eqnarray*}
\theta^2+\theta+1=3+3^{s+1}\theta^*+3^{2s}(\theta^*)^2,\ 2\theta+1=3+2\cdot3^{s}\theta^*
\end{eqnarray*}
and $\mathbf{y}^{(1)}=3(y_1+3Y_2)$ where $Y_2=y_2+3y_3+\cdots$. We consider two cases: $|q|_3<\frac{1}{9}$ and $|q|_3=\frac{1}{9}$.

\textsc{case 1}: If $|q|_3<\frac{1}{9},$ we take $\mathbf{y}=3$. Since $s>t>2$, we get
\begin{eqnarray*}
h(\mathbf{y})&=&3^3-3^2(3+3^{s+1}\theta^*+3^{2s}(\theta^*)^2)-3(3+2\cdot3^{s}\theta^*)(-3^tq^*-3^{s}\theta^*)\\ &&-(-3^tq^*-3^{s}\theta^*)^2 \\ &\equiv & 0 \ ( \rm{mod} \ 243)\\
h^{\prime}(\mathbf{y})&=&3^3-2\cdot 3(3+3^{s+1}\theta^*+3^{2s}(\theta^*)^2)-(3+2\cdot3^{s}\theta^*)(-3^tq^*-3^{s}\theta^*) \\ &\equiv & 2\cdot 3^2 \not\equiv 0 \ ( \rm{mod} \ 27).
\end{eqnarray*}
Due to Hensel's lemma there is $\mathbf{y}^{(1)} \in \mathbb{Z}_3$ such that $h(\mathbf{y}^{(1)})=0$ and $\mathbf{y}^{(1)} \equiv \mathbf{y} \ (\rm{mod} \ 27)$.

\textsc{case 2}: If $|q|_3=\frac{1}{9},$ we take $\mathbf{y}=3+3^2\bar{y}$ with $2\bar{y} \equiv (q^*)^2-q^* \ (\rm{mod} \ 3)$. Since $s>t=2$, we get
\begin{eqnarray*}
h(\mathbf{y})&=&3^3-3^2(3+3^{s+1}\theta^*+3^{2s}(\theta^*)^2)-3(3+2\cdot3^{s}\theta^*)(-3^2q^*-3^{s}\theta^*)\\ &&-(-3^2q^*-3^{s}\theta^*)^2 \\ &\equiv & 3^4\cdot2\bar{y} + 3^4q^* - 3^4(q^*)^2 \equiv 3^4\left(2\bar{y} + q^* - (q^*)^2\right) \equiv 0 \ ( \rm{mod} \ 243)\\
h^{'}(\mathbf{y})&=&3^3-2\cdot 3(3+3^{s+1}\theta^*+3^{2s}(\theta^*)^2)-(3+2\cdot3^{s}\theta^*)(-3^2q^*-3^{s}\theta^*) \\ &\equiv & 2\cdot 3^2 \not\equiv 0 \ ( \rm{mod} \ 27).
\end{eqnarray*}
Again, by Hensel's lemma there is $\mathbf{y}^{(1)} \in \mathbb{Z}_3$ such that $h(\mathbf{y}^{(1)})=0$ and $\mathbf{y}^{(1)} \equiv \mathbf{y} \ ( \rm{mod} \ 27)$.

In both cases, we have shown the existence of the root. Moreover, we have also shown $\mathbf{y}^{(1)} \equiv \mathbf{y} \equiv 3 \ ( \rm{mod} \ 9)$. Thus, $|\mathbf{y}^{(1)}-3|_3\leq \frac{1}{9}<\frac{1}{3}$.
\end{proof}

\begin{proposition}\label{fixedpointsp=3}
If $|q|_3<\frac{1}{3}$ then $\mathbf{Fix}\{f_{\theta, q, 3}\}=\{\mathbf{x}^{(0)}, \ \mathbf{x}^{(1)}\}$ where $\mathbf{x}^{(0)}:=1,\ \mathbf{x}^{(1)}:=1-q+(\theta-1)(\mathbf{y}^{(1)}-1)$ where $\mathbf{y}^{(1)}$ is the root of the cubic equation \eqref{wrty}. If $|q|_3=\frac{1}{3}$ then $\mathbf{Fix}\{f_{\theta, q, 3}\}=\{\mathbf{x}^{(0)}\}$.
\end{proposition}

\begin{proof}The proof follows directly from Proposition \ref{rootp=3}. \end{proof}

\begin{proposition}\label{behaviourfixedpointsp=3}
We have \begin{itemize}
\item[(i)] $\mathbf{x}^{(0)}$ is an attracting fixed point,
\item[(ii)] For $|q|_3<\frac{1}{3},\ $ $\mathbf{x}^{(1)}$ is a repelling fixed point.
\end{itemize}
\end{proposition}

\begin{proof}
By \eqref{derivative}, $\left|f'_{\theta, q, 3}(\mathbf{x}^{(0)})\right|_3=\frac{|\theta-1|_3}{3|q|_3} < 1$. Due to Proposition \ref{rootp=3}, we have $\left|\mathbf{y}^{(1)}\right|_3=1$ and $|\mathbf{x}^{(1)}|_3=1$. Hence, by \eqref{derivativefixed2}, $\left|f'_{\theta, q, 3}(\mathbf{x}^{(1)})\right|_3=\frac{3|q|_3}{|\theta-1|_3} > 1$.
\end{proof}

\subsubsection{Attracting basin of the attracting fixed point}
We describe
$$\mathfrak{B}(\mathbf{x}^{(0)}):=\left\{ x \in \mathbb{Q}_3 : \lim_{n \to +\infty}f^{n}_{\theta, q, 3}(x)=\mathbf{x}^{(0)}\right\}.$$

We introduce the following sets
\begin{eqnarray*}
\mathcal{A}_0 &:=& \left\{ x \in \mathbb{Q}_3 : \left|x-\mathbf{x}^{(0)}\right|_p < |q|_3 \right\}, \\ 
\mathcal{A}_1 &:=& \left\{ x \in \mathbb{Q}_3 : \left|x-\mathbf{x}^{(0)}\right|_3 > |q|_3 \right\}, \\ 
\mathcal{A}_{0,\infty} &:=& \left\{ x \in \mathbb{Q}_3 : \left|x-\mathbf{x}^{(\infty)}\right|_3 = \left|x-\mathbf{x}^{(0)}\right|_3 = |q|_3 \right\}, \\ 
\mathcal{A}_2 &:=& \left\{ x \in \mathbb{Q}_3 : \frac{\left|\theta-1\right|_3}{3} < \left|x-\mathbf{x}^{(\infty)}\right|_3 < |q|_3 \right\},\\
\mathcal{A}_{1,\infty} &:=& \left\{x \in \mathbb{Q}_3 : \left|x-\mathbf{x}^{(\infty)}\right|_3 = \frac{\left|\theta-1\right|_3}{3} \right\}, \\
\mathcal{A}_{\infty} &:=& \left\{ x \in \mathbb{Q}_3 : \ 0<\left|x-\mathbf{x}^{(\infty)}\right|_3 < \frac{\left|\theta-1\right|_3}{3} \right\}.
\end{eqnarray*}
We have $\textup{\textbf{Dom}}\{f_{\theta,q,3}\}=\mathcal{A}_0 \cup \mathcal{A}_1\cup \mathcal{A}_2\cup \mathcal{A}_{1,\infty}\cup \mathcal{A}_{0,\infty} \cup \mathcal{A}_{\infty}.$
Moreover, $\mathcal{A}_{1,\infty} = \mathcal{A}_{1,\infty}^{(1)} \cup \mathcal{A}_{1,\infty}^{(2)}$ where
\begin{eqnarray*}
\mathcal{A}_{1,\infty}^{(1)} &:=& \left\{ x \in \mathbb{Q}_3 : \left|x-\mathbf{x}^{(1)}\right|_3 = \left|x-\mathbf{x}^{(\infty)}\right|_3 = \frac{\left|\theta-1\right|_3}{3} \right\}, \\
\mathcal{A}_{1,\infty}^{(2)} &:=& \left\{x \in \mathbb{Q}_3 : \left|x-\mathbf{x}^{(1)}\right|_3 < \frac{\left|\theta-1\right|_3}{3} \right\}.
\end{eqnarray*}

Denote $\mathbf{x}^{(\infty)}=2-q-\theta.$ The following properties of the fixed points $\mathbf{x}^{(0)}$ and $\mathbf{x}^{(1)}$ (when exists, i.e., $|q|_3<\frac{1}{3}$) will be used:
\begin{itemize}
\item[(i)]$\mathbf{x}^{(1)}=\mathbf{x}^{(\infty)}+(\theta-1)\mathbf{y}^{(1)}$;
\item[(ii)] $|\mathbf{x}^{(\infty)}|_3=|\mathbf{x}^{(i)}|_3=1,\ \mathbf{x}^{(\infty)},\mathbf{x}^{(i)}\in\mathcal{E}_3$ for $i=0,1\ |\mathbf{y}^{(1)}|_3=\frac{1}{3},\ |\mathbf{y}^{(1)}-3|_3<\frac{1}{3};$ 
\item[(iii)] $|\mathbf{x}^{(0)}-\mathbf{x}^{(\infty)}|_3=|q|_3,\ |\mathbf{x}^{(1)}-\mathbf{x}^{(\infty)}|_3=\frac{|\theta-1|_3}{3}$ and $|\mathbf{x}^{(1)}-\mathbf{x}^{(0)}|_3=|q|_3;$
\item[(iv)] $\mathbf{x}^{(0)} \in \mathcal{A}_0,\ \mathbf{x}^{(1)} \in \mathcal{A}_{1,\infty}$.
\end{itemize}  

\begin{proposition}\label{p=3}
Let $0<|\theta-1|_3 < |q|_3 < 1$. Then the following hold:
\begin{itemize}
\item[(i)] $\mathcal{A}_0 \cup \mathcal{A}_1 \cup \mathcal{A}_2 \cup \mathcal{A}_{0,\infty}\subset f^{-1}_{\theta,q,3}\left(\mathcal{A}_0\right);$
\item[(ii)] $\mathcal{A}_{\infty}\subset f^{-1}_{\theta,q,3}\left(\mathcal{A}_{1} \right);$
\item[(iii)] for $|q|_3 = \frac{1}{3}$ we have $\mathcal{A}_{1,\infty}\subset f^{-1}_{\theta,q,3}\left(\mathcal{A}_{1} \right)$.
\end{itemize} 
\end{proposition}

\begin{proof}
(i) Let us first show that $f_{\theta,q,3}(x) \in \mathcal{A}_0$ for any $x \in \mathcal{A}_0 \cup \mathcal{A}_1$.
In fact, we have
$$
|x-\mathbf{x}^{(\infty)}|_3=\left|x-\mathbf{x}^{(0)}+q+(\theta-1)\right|_3=
\begin{cases}
|q|_3, & x \in \mathcal{A}_0, \\
|x-\mathbf{x}^{(0)}|_3, & x \in \mathcal{A}_1, \\
\end{cases}
$$
$$
|\theta x+q-1|_3=\left|\theta(x-\mathbf{x}^{(0)})+q+(\theta-1)\right|_3=
\begin{cases}
|q|_3, & x \in \mathcal{A}_0, \\
|x-\mathbf{x}^{(0)}|_3, & x \in \mathcal{A}_1, \\
\end{cases}
$$
which imply
$$|g(x)|_3 \leq
\begin{cases}
|q|_3^2, & x \in \mathcal{A}_0, \\
|x-\mathbf{x}^{(0)}|_3^2, & x \in \mathcal{A}_1. \\
\end{cases}
$$
Therefore, by \eqref{fminusx0} 
$$
|f_{\theta,q,3}(x)-\mathbf{x}^{(0)}|_3 \leq
\left\{
\begin{array}{cc}
\frac{|x-\mathbf{x}^{(0)}|_3}{|q|_3}|\theta-1|_3, & x \in \mathcal{A}_0 \\
|\theta-1|_3, & x \in 
\mathcal{A}_1
\end{array}
\right\}
\leq |\theta-1|_3<|q|_3.
$$

Now, let us show $f_{\theta,q,3}(x) \in \mathcal{A}_0$ for any $x \in \mathcal{A}_{0,\infty}$. In fact
\begin{eqnarray*}
|x-\mathbf{x}^{(0)}|_p&=&|x-\mathbf{x}^{(\infty)}|_3=|q|_3,\\	
|\theta x+q-1|_3&=&\left|\theta(x-\mathbf{x}^{(\infty)})+(\theta-1)(1-\theta-q)\right|_3=|q|_3.
\end{eqnarray*}
Then $|g(x)|_3\leq |q|^2_3$. Thus, by \eqref{fminusx0}, for any $x \in \mathcal{A}_{0,\infty},\ $ $|f_{\theta,q,3}(x)-\mathbf{x}^{(0)}|_3 \leq {|\theta-1|_3}<|q|_3$.

Next, we show $f_{\theta,q,3}(x) \in \mathcal{A}_0$ for any $x \in \mathcal{A}_2.$  Indeed, for any $x \in \mathcal{A}_2,$ we have 
$$
\left|\frac{(\theta-1)(1-\theta-q)}{x-\mathbf{x}^{(\infty)}}\right|_3=\frac{|q|_3|\theta-1|_3}{|x-\mathbf{x}^{(\infty)}|_3} < 3|q|_3
$$
which implies $|g_1(x)|_3 \leq \frac{1}{3}$. Therefore, by \eqref{IsingPottsSingular}
$$\left|f_{\theta, q, 3}(x)-\mathbf{x}^{(0)}\right|_3\leq\frac{\left|\theta-1\right|_3|q|_3}{3\left|x-\mathbf{x}^{(\infty)}\right|_3}<|q|_3.$$ This means $f_{\theta,q,3}(x) \in \mathcal{A}_0$ for any $x \in \mathcal{A}_0 \cup \mathcal{A}_1 \cup \mathcal{A}_{2} \cup \mathcal{A}_{0,\infty}$. Hence, the assertion follows.

(ii) We show $f_{\theta,q,3}(x) \in \mathcal{A}_1$ for any $x \in \mathcal{A}_{\infty}.$ We distinguish three cases: $|q|_3=\frac{1}{3},\ |q|_3=\frac{1}{9}$ and $|q|_3<\frac{1}{9}$.

\textsc{Case 1:} Let $|q|_3=\frac{1}{3}$. For any $x \in \mathcal{A}_{\infty}$ 
$$
\left|\frac{(\theta-1)(1-\theta-q)}{x-\mathbf{x}^{(\infty)}}\right|_3=\frac{|q|_3|\theta-1|_3}{|x-\mathbf{x}^{(\infty)}|_3}=\frac{|\theta-1|_3}{3|x-\mathbf{x}^{(\infty)}|_3}>1
$$
which implies $\left|f_{\theta,q,3}(x)\right|_3=\left|\theta+\frac{(\theta-1)(1-\theta-q)}{x-\mathbf{x}^{(\infty)}}\right|_3^3 > 1.$ Then, by \eqref{IsingPottsSingular} $$\left|f_{\theta,q,3}(x)-\mathbf{x}^{(0)}\right|_3=\left|f_{\theta,q,3}(x)\right|_3>1>|q|_3.$$ Hence, $f_{\theta,q,3}\left(\mathcal{A}_{\infty}\right)\subset\mathcal{A}_1$.

\textsc{Case 2:} Let $|q|_3=\frac{1}{9}$. Then for any $x \in \mathcal{A}_{\infty}$ such that $|x-\mathbf{x}^{(\infty)}|_3<\frac{|\theta-1|_3}{9}$ we have
$$
\left|\frac{(\theta-1)(1-\theta-q)}{x-\mathbf{x}^{(\infty)}}\right|_3=\frac{|q|_3|\theta-1|_3}{|x-\mathbf{x}^{(\infty)}|_3}=\frac{|\theta-1|_3}{9|x-\mathbf{x}^{(\infty)}|_3}>1
$$
which implies $\left|f_{\theta,q,3}(x)\right|_3=\left|\theta+\frac{(\theta-1)(1-\theta-q)}{x-\mathbf{x}^{(\infty)}}\right|_3^3 > 1.$ Thus, by \eqref{IsingPottsSingular} $$\left|f_{\theta,q,3}(x)-\mathbf{x}^{(0)}\right|_3=\left|f_{\theta,q,3}(x)\right|_3>1>|q|_3.$$

Next, we consider for any $x \in \mathcal{A}_{\infty}$ such that $|x-\mathbf{x}^{(\infty)}|_3=\frac{|\theta-1|_3}{9}.$ Let $y:=\frac{x-\mathbf{x}^{(\infty)}}{\theta -1}.$ Then, $|y|_3=|q|_3=\frac{1}{9}$ and $y=\frac{y^{*}}{|q|_p}.$ In \eqref{fminusx0},
\begin{eqnarray*}
g(x)&=&(\theta x+q-1)^2+(\theta x+q-1)(x-\mathbf{x}^{(\infty)})+(x-\mathbf{x}^{(\infty)})^2\\
&=&(\theta-1)^2\left[(\theta^2+\theta+1)y^2+(2\theta+1)(1-\theta-q)y+(1-\theta-q)^2\right]\\
&=&\frac{(\theta-1)^2}{|q|^2_3}\left[(\theta^2+\theta+1)(y^{*})^2+(2\theta+1)(1-\theta-q)^{*}(y^{*})+\left((1-\theta-q)^{*}\right)^2\right].
\end{eqnarray*} Hence, $|g(x)|_3=|q|^2_3|\theta-1|^2_3$ for any $x \in \mathcal{A}_{\infty}$. By \eqref{fminusx0}, $|f_{\theta,q,3}(x)-\mathbf{x}^{(0)}|_3=1>|q|_3$ for any $x\in \mathcal{A}_{\infty}$ which means $f_{\theta,q,3}\left(\mathcal{A}_{\infty}\right)\subset\mathcal{A}_1.$

\textsc{Case 3:} Let $|q|_3<\frac{1}{9}$. Note $\mathcal{A}_\infty = \mathcal{A}_{\infty}^{(1)} \cup \mathcal{A}_{\infty}^{(2)}$ where
\begin{eqnarray*}
\mathcal{A}_{\infty}^{(1)} &:=& \left\{ x \in \mathbb{Q}_3 : 3|q|_3{\left|\theta-1\right|_3} \leq \left|x-\mathbf{x}^{(\infty)}\right|_3 < \frac{\left|\theta-1\right|_3}{3} \right\}, \\
\mathcal{A}_{\infty}^{(2)} &:=& \left\{x \in \mathbb{Q}_3 : \left|x-\mathbf{x}^{(\infty)}\right|_3 < 3|q|_3{\left|\theta-1\right|_3} \right\}.
\end{eqnarray*}

We need to show that for any $x \in \mathcal{A}_{\infty}^{(1)}\left(\mathcal{A}_{\infty}^{(2)}\right),\ f_{\theta,q,3}(x) \in \mathcal{A}_1.$ Suppose $x \in \mathcal{A}_{\infty}^{(1)}.$ We recall $g_1$ in \eqref{IsingPottsSingular2} and write it as
\begin{eqnarray}\label{g_x}
g_1(x)&=&\left(\theta+\frac{(\theta-1)(1-\theta-q)}{x-\mathbf{x}^{(\infty)}}\right)^2+\left(\theta+\frac{(\theta-1)(1-\theta-q)}{x-\mathbf{x}^{(\infty)}}\right)+1 \nonumber \\
&=& \left(\theta^2+\theta+1\right)+(2\theta+1)\left(\frac{(\theta-1)(1-\theta-q)}{x-\mathbf{x}^{(\infty)}}\right) \nonumber \\
&& +\left(\frac{(\theta-1)(1-\theta-q)}{x-\mathbf{x}^{(\infty)}}\right)^2.
\end{eqnarray}
We need the fact
\begin{eqnarray}
\left|\frac{(\theta-1)(1-\theta-q)}{x-\mathbf{x}^{(\infty)}}\right|_3=\frac{|\theta-1|_3|q|_3}{|x-\mathbf{x}^{(\infty)}|_3}\leq\frac{1}{3},\quad |\theta^2+\theta+1|_p=\frac{1}{3} ,\quad |2\theta+1|_3 = \frac{1}{3} \nonumber
\end{eqnarray}
which imply, by \eqref{g_x}, $\left|g_1(x)\right|_3 = \frac{1}{3}.$ Note $|x-\mathbf{x}^{(0)}|_3=|q|_3$.
Then, by $3|q|_3{\left|\theta-1\right|_3} \leq \left|x-\mathbf{x}^{(\infty)}\right|_3 < \frac{\left|\theta-1\right|_3}{3}$ and \eqref{IsingPottsSingular2}, $$|q|_3 < \left|f_{\theta,q,3}(x)-\mathbf{x}^{(0)}\right|_3= \frac{|\theta-1|_3|x-\mathbf{x}^{(0)}|_3}{|x-\mathbf{x}^{(\infty)}|_3}|g_1(x)|_3 \leq \frac{1}{9}.$$ Therefore, $f_{\theta,q,3}(x) \in \mathcal{A}_1.$

Suppose $x \in \mathcal{A}_{\infty}^{(2)}.$ We consider $|x-\mathbf{x}^{(\infty)}|_3<{|q|_3|\theta-1|_3}.$ Then we have $$
\left|\frac{(\theta-1)(1-\theta-q)}{x-\mathbf{x}^{(\infty)}}\right|_3=\frac{|q|_3|\theta-1|_3}{|x-\mathbf{x}^{(\infty)}|_3}>1
$$ which implies by \eqref{IsingPottsSingular}, $\left|f_{\theta,q,3}(x)\right|_3=\left|\theta+\frac{(\theta-1)(1-\theta-q)}{x-\mathbf{x}^{(\infty)}}\right|_3^3 > 1 > |q|_3.$ Therefore, since $|\mathbf{x}^{(0)}|_3=1$ we have $\left|f_{\theta,q,3}(x)-\mathbf{x}^{(0)}\right|_3=\left|f_{\theta,q,3}(x)\right|_3>1>|q|_3$.

Now, consider for $|x-\mathbf{x}^{(\infty)}|_3={|q|_3|\theta-1|_3}.$ Let $y=\frac{x-\mathbf{x}^{(\infty)}}{\theta -1}.$ Then, $|y|_3=|q|_3.$ Write $y=\frac{y^{*}}{|q|_3}$ and in \eqref{fminusx0} we have
\begin{eqnarray*}
g(x)&=&(\theta x+q-1)^2+(\theta x+q-1)(x-\mathbf{x}^{(\infty)})+(x-\mathbf{x}^{(\infty)})^2\\
&=&(\theta-1)^2\left[(\theta^2+\theta+1)y^2+(2\theta+1)(1-\theta-q)y+(1-\theta-q)^2\right]\\
&=&\frac{(\theta-1)^2}{|q|^2_3}\left[(\theta^2+\theta+1)(y^{*})^2+(2\theta+1)(1-\theta-q)^{*}(y^{*})+\left((1-\theta-q)^{*}\right)^2\right].
\end{eqnarray*}
Then, $|g(x)|_3=|q|^2_3|\theta-1|^2_3$ for any $x \in \mathcal{A}_{\infty}^{(2)}$. Hence, by \eqref{fminusx0}, $|f_{\theta,q,3}(x)-\mathbf{x}^{(0)}|_3=1>|q|_3$ for any $x\in \mathcal{A}_{\infty}^{(2)}$ which implies $f_{\theta,q,3}\left(\mathcal{A}_{\infty}\right)\subset\mathcal{A}_1.$

(iii) Suppose $|q|_3=\frac{1}{3}$. Let $x \in \mathcal{A}_{1,\infty}$. We have
\begin{eqnarray*}
|\theta^2+\theta+1|_3 = \frac{1}{3}, \quad |2\theta+1|_3 = \frac{1}{3}, \quad \frac{|\theta-1|_3|q|_3}{|x-\mathbf{x}^{(\infty)}|_3} = 1.
\end{eqnarray*}
Then, by \eqref{g_x}, $\left|g_1(x)\right|_3 = 1$ and by \eqref{IsingPottsSingular2} $\left|f_{\theta,q,3}(x)-\mathbf{x}^{(0)}\right|_3=\frac{\left|\theta-1\right|_3\left|x-\mathbf{x}^{(0)}\right|_3}{\left|x-\mathbf{x}^{(\infty)}\right|_3} = 1 > |q|_3.$ Therefore, $f_{\theta,q,3}(x) \in \mathcal{A}_1$ and the assertion holds. 
\end{proof}

\begin{proposition}\label{p=3_1}
If $|q|_3 < \frac{1}{3},$ then $\mathcal{A}_{1,\infty}^{(1)}\subset f^{-1}_{\theta,q,3}\left(\mathcal{A}_{0,\infty} \right)$. Moreover, for any $x,y \in \mathcal{A}_{1,\infty}^{(2)}$,
\begin{itemize}
\item[(i)] $\left|f_{\theta,q,3}(x)-f_{\theta,q,3}(y)\right|_3 = \frac{3|q|_3}{|\theta-1|_3}|x-y|_3,$
\item[(ii)] there exist $n_0 \in \mathbb{N}$ such that $f^{n_0}_{\theta,q,3}(x) \notin \mathcal{A}_{1,\infty}^{(2)}$.
\end{itemize}
\end{proposition}

\begin{proof}
We prove $f_{\theta,q,3}(x) \in \mathcal{A}_{0,\infty}$ for any $x \in \mathcal{A}_{1,\infty}^{(1)}.$

Suppose $x \in \mathcal{A}_{1,\infty}^{(1)}.$ We have
\begin{eqnarray*}
\left|\frac{(\theta-1)(1-\theta-q)}{x-\mathbf{x}^{(\infty)}}\right|_3=\frac{|\theta-1|_3|q|_3}{|x-\mathbf{x}^{(\infty)}|_3}\leq\frac{1}{3}, \quad
|2\theta+1|_3 = \frac{1}{3}
\end{eqnarray*}
which imply by \eqref{g_x}, $\left|g_1(x)\right|_3 = \frac{1}{3}$. Hence, by \eqref{IsingPottsSingular2} $$\left|f_{\theta,q,3}(x)-\mathbf{x}^{(0)}\right|_3=\frac{\left|\theta-1\right|_3\left|x-\mathbf{x}^{(0)}\right|_3}{\left|x-\mathbf{x}^{(\infty)}\right|_3}|g_1(x)|_3 = |q|_3.$$

Since $|\theta-1|_3<|q|_3$ and $|f_{\theta,q,3}(x)-\mathbf{x}^{(\infty)}|_3=|f_{\theta,q,3}(x)-\mathbf{x}^{(0)}+q+(\theta-1)|_3,$ it suffices to show $|f_{\theta,q,3}(x)-\mathbf{x}^{(0)}+q|_3=|q|_3.$ We have
\begin{multline*}
f_{\theta,q,3}(x)-\mathbf{x}^{(0)}+q=\frac{(\theta-1)(x-\mathbf{x}^{(0)})}{x-\mathbf{x}^{(\infty)}}\left(g_1(x)-3\right)\\+3\left((\theta-1)-\frac{(\theta-1)^2}{x-\mathbf{x}^{(\infty)}}\right)+\frac{q(\theta-1)(\mathbf{y}^{(1)}-3)}{x-\mathbf{x}^{(\infty)}}+q\frac{x-\mathbf{x}^{(1)}}{x-\mathbf{x}^{(\infty)}}.
\end{multline*}
Since $|\mathbf{y}^{(1)}-3|_3<\frac{1}{3},\ |x-\mathbf{x}^{(0)}|_3=|q|_3,\ |x-\mathbf{x}^{(\infty)}|_3=|x-\mathbf{x}^{(1)}|_3=\frac{|\theta-1|_3}{3},\ |g_1(x)-3|_3<\frac{1}{3},$ we have $|f_{\theta,q,3}(x)-\mathbf{x}^{(0)}+q|_3=|q|_3.$ This means $|f_{\theta,q,3}(x)-\mathbf{x}^{(\infty)}|_3=|q|_3$. Therefore, $\mathcal{A}_{1,\infty}^{(1)}\subset f^{-1}_{\theta,q,3}\left(\mathcal{A}_{0,\infty} \right).$

Now, we show the other two assertions.
(i) Suppose $\bar{x},\bar{\bar{x}}\in\mathcal{A}_{1,\infty}^{(2)}.$ We have
\begin{eqnarray*}
\left|3\theta(1-\theta-q)\left(\frac{\theta-1}{\bar{x}-\mathbf{x}^{\infty}}+\frac{\theta-1}{\bar{\bar{x}}-\mathbf{x}^{\infty}}\right)\right|_3\leq |q|_3 < \frac{1}{3},\\
\left|(1-\theta-q)^2\left(\frac{(\theta-1)^2}{\left(\bar{x}-\mathbf{x}^{\infty}\right)^2}+\frac{(\theta-1)^2}{\left(\bar{x}-\mathbf{x}^{\infty}\right)\left(\bar{\bar{x}}-\mathbf{x}^{\infty}\right)}+\frac{(\theta-1)^2}{\left(\bar{\bar{x}}-\mathbf{x}^{\infty}\right)^2}\right) \right|_3\leq(3|q|_3)^2 < \frac{1}{3}
\end{eqnarray*}
which imply $|F_{\theta, q, 3}(\bar{x},\bar{\bar{x}})|_3=\frac{1}{3}$. Thus, by \eqref{distancexy}
$$\left|f_{\theta, q, 3}(\bar{x}) -f_{\theta, q, 3}(\bar{\bar{x}})\right|_3=\frac{3|q|_3}{|\theta-1|_3}|\bar{x}-\bar{\bar{x}}|_3.$$

The assertion (ii) follows from the fact $\mathbf{x}^{(1)} \in \mathcal{A}_{1,\infty}^{(2)}$ and (i).
\end{proof}

We conclude the results in this subsection by the following theorem.

\begin{theorem}
We describe the attracting basin of $\mathbf{x}^{(0)}$ as follows.
\begin{itemize}
\item[(i)] If $|q|_3=\frac{1}{3},$ then
$$\mathfrak{B}(\mathbf{x}^{(0)})= \mathbb{Q}_3 \setminus \left(\bigcup_{n=0}^{+\infty}f^{-n}_{\theta, q, 3}\left\{\mathbf{x}^{(\infty)}\right\}\right).$$
\item[(ii)] If $|q|_3<\frac{1}{3},$ then
$$\mathfrak{B}(\mathbf{x}^{(0)}):= \mathbb{Q}_3 \setminus \left(\mathbf{x}^{(1)} \cup \bigcup_{n=0}^{+\infty}f^{-n}_{\theta, q, 3}\left\{\mathbf{x}^{(\infty)}\right\}\right).$$
\end{itemize}
\end{theorem}

\begin{proof}
It follows from Propositions \ref{p=3} and \ref{p=3_1}.
\end{proof}

\subsection{Case \texorpdfstring{$p=2$}{Lg}}

\subsubsection{Fixed points}
Let us find all possible roots of the cubic equation \eqref{wrty}. We always assume $0<|\theta-1|_2<|q|_2<1.$
\begin{proposition}\label{rootp=2}
The cubic equation \eqref{wrty} always has one root $\mathbf{y}^{(1)}$ over $\mathbb{Q}_2$ such that $|\mathbf{y}^{(1)}|_2 = 1.$
\end{proposition}

\begin{proof}
Let $|q|_2=2^{-t},\ |\theta-1|_2=2^{-s}.$ Then $s>r>0$. Write $q = 2^tq^{*}$ and $\theta = 1 + 2^s\theta^*$ for which $|q^{*}|_2=1,\ |\theta^*|_2=1$. In the cubic equation \eqref{wrty}, we let $a_0=-(1-\theta-q)^2,\ a_1=-(2\theta+1)(1-\theta-q),\ a_2=-(1+\theta+\theta^2)$ and $a_3=1$. Then $|a_0|_2=|q|_2^2,\ |a_1|_2=|q|_2,\ |a_2|_2=1,\ |a_3|_2=1$ and $-\log_2|a_0|_2=2t,\ -\log_2|a_1|_2=t,\ -\log_2|a_2|_2=0,\ -\log_2|a_3|_2=0$. We consider the following points on the plane, $(0,2t),\ (1,t),\ (2,0),\ (3,0)$ and construct the corresponding Newton polygon. We have the line from $(0,2t)$ to $(2,0)$ and from $(2,0)$ to $(3,0)$. Thus we have the slope $m_1=0$ with  length $l_1=1$ and the slope $m_2=-t$ with  length $l_2=2$. Thus, we have three roots $\mathbf{y}^{(1)},\ \mathbf{y}^{(2)},\ \mathbf{y}^{(3)} \in \mathbb{C}_2$ (the field of $2$-adic complex numbers) with $|\mathbf{y}^{(1)}|_2=1$ and $|\mathbf{y}^{(i)}|_2=2^{-r}$ for $i=2,3$ (including multiplicity).

Next, we show $\mathbf{y}^{(1)} \in \mathbb{Q}_2$ and $\mathbf{y}^{(2)},\mathbf{y}^{(3)} \not\in \mathbb{Q}_2$. Let us consider the following cubic polynomial and its derivative,
\begin{eqnarray*}
h(y) &=& y^3-(1+\theta+\theta^2)y^2-(2\theta+1)(1-\theta-q)y-(1-\theta-q)^2 \\
h^{\prime}(y) &=& 3y^2-2(1+\theta+\theta^2)y-(2\theta+1)(1-\theta-q).
\end{eqnarray*}
Let $\bar{\mathbf{y}}=1$. Then we get $h(\bar{\mathbf{y}})\equiv 0 \ (\rm{mod} \ 2)$ and $h^{\prime}(\bar{\mathbf{y}})\not\equiv 0 \ (\rm{mod} \ 2)$. According to Hensel's lemma, there is $\mathbf{y}^{(1)} \in \mathbb{Z}_2 \subset \mathbb{Q}_2$ such that $h(\mathbf{y}^{(1)})=0$ where $\mathbf{y}^{(1)} \equiv \bar{\mathbf{y}} \equiv 1 \ (\rm{mod} \ 2)$.

Now, suppose $\bar{\mathbf{y}} = 2^{r}y^{*}$ with $|y^{*}|_2=1$. We have
\begin{eqnarray*}
h(\bar{\mathbf{y}}) &=& 2^{3r}\left(y^{*}\right)^3 - 2^{2r}(1+\theta+\theta^2)\left(y^{*}\right)^2-2^{2r}(2\theta+1)(1-\theta-q)^{*}\left(y^{*}\right) \\ & & -\left((1-\theta-q)^{*}\right)^2 \\
& \equiv &  - 2^{2r}(1+\theta+\theta^2)\left(y^{*}\right)^2-2^{2r}(2\theta+1)(1-\theta-q)^{*}\left(y^{*}\right) \\ && - 2^{2r}\left((1-\theta-q)^{*}\right)^2 \ (\rm{mod} \ 2^{2r+1}) \\
& \equiv &  -(1+\theta+\theta^2)\left(y^{*}\right)^2-(2\theta+1)(1-\theta-q)^{*}\left(y^{*}\right) -\left((1-\theta-q)^{*}\right)^2 \ (\rm{mod} \ 2) \\
& \not\equiv & 0 \ (\rm{mod} \ 2).
\end{eqnarray*}
Then $h(\bar{\mathbf{y}}) \neq 0$ for any $\bar{\mathbf{y}} \in \mathbb{Q}_2,\ |\bar{\mathbf{y}}|_2=|q|_2=2^{-r}$. Since $|{\mathbf{y}^{(2)}}|_2=|{\mathbf{y}^{(3)}}|_2=|q|_2=2^{-r}$, we have $\mathbf{y}^{(2)},\mathbf{y}^{(3)} \not\in \mathbb{Q}_2$. Therefore, the cubic equation \eqref{wrty} always has one root $\mathbf{y}^{(1)}$ over $\mathbb{Q}_2$ 
\end{proof}

\begin{proposition}\label{fixedpointsp=2}
We have $\mathbf{Fix}\{f_{\theta, q, 3}\}=\{\mathbf{x}^{(0)}, \ \mathbf{x}^{(1)}\}$ where $\mathbf{x}^{(0)}=1,\ \mathbf{x}^{(1)}=1-q+(\theta-1)(\mathbf{y}^{(1)}-1)$ and $\mathbf{y}^{(1)}$ is a root of the cubic equation \eqref{wrty}.
\end{proposition}

\begin{proof}The proof follows directly from Proposition \ref{rootp=2}. \end{proof}

\begin{proposition}\label{behaviourfixedpointsp=2}
We have
\begin{itemize}
\item[(i)] $\mathbf{x}^{(0)}$ is an attracting fixed point;
\item[(ii)] $\mathbf{x}^{(1)}$ is a repelling fixed points.
\end{itemize}
\end{proposition}

\begin{proof}
By \eqref{derivative}, $\left|f'_{\theta, q, 3}(\mathbf{x}^{(0)})\right|_2=\frac{|\theta-1|_2}{|q|_2} < 1$. Due to Proposition \ref{rootp=2}, we have $\left|\mathbf{y}^{(1)}\right|_2=1$ and $|\mathbf{x}^{(1)}|_2=1$. Therefore, by \eqref{derivativefixed2} we get $\left|f'_{\theta, q, 3}(\mathbf{x}^{(1)})\right|_2=\frac{|q|_2}{|\theta-1|_2} > 1$.
\end{proof}

\subsubsection{Attracting basin of the attracting fixed point}
We describe $$\mathfrak{B}(\mathbf{x}^{(0)}):=\left\{ x \in \mathbb{Q}_2 : \lim_{n \to +\infty}f^{n}_{\theta, q, 3}(x)=\mathbf{x}^{(0)}\right\}.$$

We introduce the following sets,
\begin{eqnarray*}
\mathcal{A}_0 &:=& \left\{ x \in \mathbb{Q}_2 : \left|x-\mathbf{x}^{(0)}\right|_2 < |q|_2 \right\}, \\ 
\mathcal{A}_1 &:=& \left\{ x \in \mathbb{Q}_2 : \left|x-\mathbf{x}^{(0)}\right|_2 > |q|_2 \right\}, \\ 
\mathcal{A}_{0,\infty} &:=& \left\{ x \in \mathbb{Q}_2 : \left|x-\mathbf{x}^{(\infty)}\right|_2 = \left|x-\mathbf{x}^{(0)}\right|_2 = |q|_2 \right\}, \\ 
\mathcal{A}_2 &:=& \left\{ x \in \mathbb{Q}_2 : \left|\theta-1\right|_2 < \left|x-\mathbf{x}^{(\infty)}\right|_2 < |q|_2 \right\},\\
\mathcal{A}_{1,\infty}^{(1)} &:=& \left\{ x \in \mathbb{Q}_2 : \left|x-\mathbf{x}^{(1)}\right|_2 = \left|x-\mathbf{x}^{(\infty)}\right|_2 = |\theta-1|_2 \right\}, \\ 
\mathcal{A}_{1,\infty}^{(2)} &:=& \left\{ x \in \mathbb{Q}_2 : \left|x-\mathbf{x}^{(1)}\right|_2 < \left|x-\mathbf{x}^{(\infty)}\right|_2 = |\theta-1|_2 \right\}, \\
\mathcal{A}_{\infty} &:=& \left\{ x \in \mathbb{Q}_2 : \ 0<\left|x-\mathbf{x}^{(\infty)}\right|_2 < \left|\theta-1\right|_2 \right\}.
\end{eqnarray*}
The following properties of the fixed points $\mathbf{x}^{(0)}, \mathbf{x}^{(1)}$ will be used.
\begin{itemize}
	\item[(i)] $\mathbf{x}^{(0)}=1$ and $\mathbf{x}^{(\infty)}=2-q-\theta=\mathbf{x}^{(0)}-q-(\theta-1);$
	\item[(ii)] $\mathbf{x}^{(1)}=1-q+(\theta-1)(\mathbf{y}^{(1)}-1)=\mathbf{x}^{(\infty)}+(\theta-1)\mathbf{y}^{(1)}$;
	\item[(iii)] $|\mathbf{x}^{(\infty)}|_2=|\mathbf{x}^{(i)}|_2=1,\ \mathbf{x}^{(\infty)},\mathbf{x}^{(i)}\in\mathcal{E}_2$ for $i=0,1\ |\mathbf{y}^{(1)}|_2=1,\ |\mathbf{y}^{(1)}-3|_2<1;$ 
	\item[(iv)] $|\mathbf{x}^{(0)}-\mathbf{x}^{(\infty)}|_2=|q|_2,\ |\mathbf{x}^{(1)}-\mathbf{x}^{(\infty)}|_2=|\theta-1|_2$ and $|\mathbf{x}^{(1)}-\mathbf{x}^{(0)}|_2=|q|_2.$
	\item[(v)] $\mathbf{x}^{(0)} \in \mathcal{A}_0,\ \mathbf{x}^{(1)} \in \mathcal{A}_{1,\infty}^{(2)}$.
\end{itemize}

\begin{proposition}\label{p=2}
The following inclusions hold:
\begin{itemize}
\item[(i)] $\mathcal{A}_0 \cup \mathcal{A}_1 \cup \mathcal{A}_2 \cup \mathcal{A}_{0,\infty}\subset f^{-1}_{\theta,q,3}\left(\mathcal{A}_0\right);$
\item[(ii)] $\mathcal{A}_{\infty}\subset f^{-1}_{\theta,q,3}\left(\mathcal{A}_{1} \right);$
\item[(iii)] $\mathcal{A}_{1,\infty}^{(1)}\subset f^{-1}_{\theta,q,3}\left(\mathcal{A}_{1} \right).$
\end{itemize} 
\end{proposition}

\begin{proof}
(i) Let us first show that $f_{\theta,q,3}(x) \in \mathcal{A}_0$ for any $x \in \mathcal{A}_0 \cup \mathcal{A}_1$.
In fact, we have
$$
|x-\mathbf{x}^{(\infty)}|_2=\left|x-\mathbf{x}^{(0)}+q+(\theta-1)\right|_2=
\begin{cases}
|q|_2, & x \in \mathcal{A}_0, \\
|x-\mathbf{x}^{(0)}|_2, & x \in \mathcal{A}_1, \\
\end{cases}
$$
$$
|\theta x+q-1|_2=\left|\theta(x-\mathbf{x}^{(0)})+q+(\theta-1)\right|_2=
\begin{cases}
|q|_2, & x \in \mathcal{A}_0, \\
|x-\mathbf{x}^{(0)}|_2, & x \in \mathcal{A}_1, \\
\end{cases}
$$
which imply
$$|g(x)|_2 \leq
\begin{cases}
|q|_p^2, & x \in \mathcal{A}_0, \\
|x-\mathbf{x}^{(0)}|_2^2, & x \in \mathcal{A}_1. \\
\end{cases}
$$
Then, by \eqref{fminusx0} 
$$
|f_{\theta,q,3}(x)-\mathbf{x}^{(0)}|_2 \leq
\left\{
\begin{array}{cc}
\frac{|x-\mathbf{x}^{(0)}|_2}{|q|_2}|\theta-1|_2, & x \in \mathcal{A}_0 \\
|\theta-1|_2, & x \in 
\mathcal{A}_1
\end{array}
\right\}
\leq |\theta-1|_2<|q|_2.
$$
Thus, $f_{\theta,q,3}(x) \in \mathcal{A}_0$ for any $x \in \mathcal{A}_0 \cup \mathcal{A}_1$.

We then show $f_{\theta,q,3}(x) \in \mathcal{A}_0$ for any $x \in \mathcal{A}_{0,\infty}$. We get
\begin{eqnarray*}
|x-\mathbf{x}^{(0)}|_2=|x-\mathbf{x}^{(\infty)}|_2 =|q|_2,\\	
|\theta x+q-1|_2 = \left|\theta(x-\mathbf{x}^{(\infty)})+(\theta-1)(1-\theta-q)\right|_2= |q|_2
\end{eqnarray*}
which imply $|g(x)|_2  \leq |q|^2_2$. Then, by \eqref{fminusx0}, $|f_{\theta,q,3}(x)-\mathbf{x}^{(0)}|_2 \leq {|\theta-1|_2}<|q|_2$ for any $x \in \mathcal{A}_{0,\infty}$. Hence, for any $x \in \mathcal{A}_{0,\infty},\ $ $f_{\theta,q,3}(x) \in \mathcal{A}_0.$

Finally, we show $f_{\theta,q,3}(x) \in \mathcal{A}_0$ for any $x \in \mathcal{A}_2.$  Indeed, for any $x \in \mathcal{A}_2,$ we have
$$
\left|\frac{(\theta-1)(1-\theta-q)}{x-\mathbf{x}^{(\infty)}}\right|_2=\frac{|q|_2|\theta-1|_2}{|x-\mathbf{x}^{(\infty)}|_2}<1
$$
which implies $|g_1(x)|_2 = 1$. Then $$\left|f_{\theta, q, 3}(x)-\mathbf{x}^{(0)}\right|_2=\frac{\left|\theta-1\right|_2}{\left|x-\mathbf{x}^{(\infty)}\right|_2}\left|x-\mathbf{x}^{(0)}\right|_2<\left|x-\mathbf{x}^{(0)}\right|_2=|q|_2.$$ Therefore, for any $x \in \mathcal{A}_2,\ $ $f_{\theta,q,3}(x) \in \mathcal{A}_0$ and the assertion (i) holds.

(ii) We want to show that $f_{\theta,q,3}(x) \in \mathcal{A}_1$ for any $x \in \mathcal{A}_{\infty}.$ Note $\mathcal{A}_{\infty} = \mathcal{A}_{\infty}^{(1)} \cup \mathcal{A}_{\infty}^{(2)} \cup \mathcal{A}_{\infty}^{(3)}$ where
\begin{eqnarray*}
\mathcal{A}_{\infty}^{(1)} &:=& \left\{ x \in \mathbb{Q}_2 : \ |q|_2\left|\theta-1\right|_2<\left|x-\mathbf{x}^{(\infty)}\right|_2 < \left|\theta-1\right|_2 \right\}, \\
\mathcal{A}_{\infty}^{(2)} &:=& \left\{ x \in \mathbb{Q}_2 : \ \left|x-\mathbf{x}^{(\infty)}\right|_2 = |q|_2\left|\theta-1\right|_2 \right\}, \\
\mathcal{A}_{\infty}^{(3)} &:=& \left\{ x \in \mathbb{Q}_2 : \ 0<\left|x-\mathbf{x}^{(\infty)}\right|_2 < |q|_2\left|\theta-1\right|_2 \right\}.
\end{eqnarray*}
Let $x \in \mathcal{A}_{\infty}^{(3)}$. By \eqref{IsingPottsSingular} 
$$
\left|\frac{(\theta-1)(1-\theta-q)}{x-\mathbf{x}^{(\infty)}}\right|_2=\frac{|q|_2|\theta-1|_2}{|x-\mathbf{x}^{(\infty)}|_2}>1
$$
which implies $\left|f_{\theta,q,3}(x)\right|_2=\left|\theta+\frac{(\theta-1)(1-\theta-q)}{x-\mathbf{x}^{(\infty)}}\right|_2^3 > 1$. Then $\left|f_{\theta,q,3}(x)-\mathbf{x}^{(0)}\right|_2=\left|f_{\theta,q,3}(x)\right|_2>1>|q|_2.$ So we have $f_{\theta,q,3}\left(\mathcal{A}_{\infty}^{(3)}\right)\subset\mathcal{A}_1$. 

Let $x \in \mathcal{A}_{\infty}^{(1)}$. We have $\left|\frac{(\theta-1)(1-\theta-q)}{x-\mathbf{x}^{(\infty)}}\right|_2=\frac{|\theta-1|_2|q|_2}{|x-\mathbf{x}^{(\infty)}|_2}<1$ which implies $\left|g_1(x)\right|_2= 1.$ Therefore, by \eqref{IsingPottsSingular2} $$|q|_2<\left|f_{\theta,q,3}(x)-\mathbf{x}^{(0)}\right|_2=\frac{\left|\theta-1\right|_2\left|x-\mathbf{x}^{(0)}\right|_2}{\left|x-\mathbf{x}^{(\infty)}\right|_2}<1$$ and $f_{\theta,q,3}\left(\mathcal{A}_{\infty}^{(1)}\right)\subset\mathcal{A}_1.$

Suppose $x \in \mathcal{A}_{\infty}^{(2)}$ and $y=\frac{x-\mathbf{x}^{(\infty)}}{\theta -1}$ such that $|y|_2=|q|_2.$ Write $y=\frac{y^{*}}{|q|_2}.$ Then
\begin{eqnarray*}
g(x)&=&(\theta x+q-1)^2+(\theta x+q-1)(x-\mathbf{x}^{(\infty)})+(x-\mathbf{x}^{(\infty)})^2\\
&=&(\theta-1)^2\left[(\theta^2+\theta+1)y^2+(2\theta+1)(1-\theta-q)y+(1-\theta-q)^2\right]\\
&=&\frac{(\theta-1)^2}{|q|^2_2}\left[(\theta^2+\theta+1)(y^{*})^2+(2\theta+1)(1-\theta-q)^{*}(y^{*})+\left((1-\theta-q)^{*}\right)^2\right].
\end{eqnarray*}
Thus, $|g(x)|_2=|q|^2_2|\theta-1|^2_2$ for any $x \in \mathcal{A}_{\infty}.$ By \eqref{fminusx0}, $|f_{\theta,q,3}(x)-\mathbf{x}^{(0)}|_2=1>|q|_2$ for any $x\in \mathcal{A}_{\infty}^{(2)}$. Hence, $f_{\theta,q,3}\left(\mathcal{A}_{\infty}^{(2)}\right)\subset\mathcal{A}_1.$ Therefore, we have shown $\mathcal{A}_{\infty}\subset f^{-1}_{\theta,q,3}\left(\mathcal{A}_{1} \right).$

(iii) We prove $f_{\theta,q,3}(x) \in \mathcal{A}_{0,\infty}$ for any $x \in \mathcal{A}_{1,\infty}^{(1)}.$ 

We have
$\left|\frac{(\theta-1)(1-\theta-q)}{x-\mathbf{x}^{(\infty)}}\right|_2=\frac{|\theta-1|_2|q|_2}{|x-\mathbf{x}^{(\infty)}|_2}<1$ which implies $\left|g_1(x)-3\right|_2<\left|g_1(x)\right|_2= 1$. Then, by \eqref{IsingPottsSingular2} $$\left|f_{\theta,q,3}(x)-\mathbf{x}^{(0)}\right|_2=\frac{\left|\theta-1\right|_2\left|x-\mathbf{x}^{(0)}\right|_2}{\left|x-\mathbf{x}^{(\infty)}\right|_2}=|q|_2.$$

We now show $|f_{\theta,q,3}(x)-\mathbf{x}^{(\infty)}|_2=|q|_2$ for any  $x \in \mathcal{A}_{1,\infty}^{(1)}.$ Since  $|f_{\theta,q,3}(x)-\mathbf{x}^{(\infty)}|_2=|f_{\theta,q,3}(x)-\mathbf{x}^{(0)}+q+(\theta-1)|_2=|f_{\theta,q,3}(x)-\mathbf{x}^{(0)}+q|_2,$ it suffices to show $|f_{\theta,q,3}(x)-\mathbf{x}^{(0)}+q|_2=|q|_2.$ We have
\begin{multline*}
f_{\theta,q,3}(x)-\mathbf{x}^{(0)}+q=\frac{(\theta-1)(x-\mathbf{x}^{(0)})}{x-\mathbf{x}^{(\infty)}}\left(g_1(x)-3\right)\\+3\left((\theta-1)-\frac{(\theta-1)^2}{x-\mathbf{x}^{(\infty)}}\right)+\frac{q(\theta-1)(\mathbf{y}^{(1)}-3)}{x-\mathbf{x}^{(\infty)}}+q\frac{x-\mathbf{x}^{(1)}}{x-\mathbf{x}^{(\infty)}}.
\end{multline*}
We have $|\mathbf{y}^{(1)}-3|_2<1,\ |x-\mathbf{x}^{(0)}|_2=|q|_2,\ |x-\mathbf{x}^{(\infty)}|_2=|x-\mathbf{x}^{(1)}|_2=|\theta-1|_2, \ |g_1(x)-3|_2<1.$ Then, $|f_{\theta,q,3}(x)-\mathbf{x}^{(0)}+q|_2=|q|_2.$ Then $|f_{\theta,q,3}(x)-\mathbf{x}^{(\infty)}|_2=|q|_2$. Thus, $f_{\theta,q,3}\left(\mathcal{A}_{1,\infty}^{(1)}\right)\subset\mathcal{A}_{0,\infty}.$
\end{proof}

\begin{proposition} \label{p=2_1}
For any $x,y \in \mathcal{A}_{1,\infty}^{(2)}$,
\begin{itemize}
\item[(i)] $\left|f_{\theta,q,3}(x)-f_{\theta,q,3}(y)\right|_2 = \frac{|q|_2}{|\theta-1|_2}|x-y|_2,$
\item[(ii)] there is $n_0 \in \mathbb{N}$ such that $f^{n_0}_{\theta,q,3}(x) \notin \mathcal{A}_{1,\infty}^{(2)}$.
\end{itemize}
\end{proposition}

\begin{proof}
(i) Suppose $\bar{x},\bar{\bar{x}}\in\mathcal{A}_{1,\infty}^{(2)}.$ Then
\begin{eqnarray*}
\left|3\theta(1-\theta-q)\left(\frac{\theta-1}{\bar{x}-\mathbf{x}^{\infty}}+\frac{\theta-1}{\bar{\bar{x}}-\mathbf{x}^{\infty}}\right)\right|_2 \leq |q|_2,\\
\left|(1-\theta-q)^2\left(\frac{(\theta-1)^2}{\left(\bar{x}-\mathbf{x}^{\infty}\right)^2}+\frac{(\theta-1)^2}{\left(\bar{x}-\mathbf{x}^{\infty}\right)\left(\bar{\bar{x}}-\mathbf{x}^{\infty}\right)}+\frac{(\theta-1)^2}{\left(\bar{\bar{x}}-\mathbf{x}^{\infty}\right)^2}\right) \right|_2\leq(|q|_2)^2
\end{eqnarray*}
which imply $|F_{\theta, q, 3}(\bar{x},\bar{\bar{x}})|_2=1$. Thus, by \eqref{distancexy}, $\left|f_{\theta, q, 3}(\bar{x}) -f_{\theta, q, 3}(\bar{\bar{x}})\right|_2=\frac{|q|_2}{|\theta-1|_2}|\bar{x}-\bar{\bar{x}}|_2$.

By (i) and the fact $\mathbf{x}^{(1)} \in \mathcal{A}_{1,\infty}^{(2)}$, we get the assertion (ii).
\end{proof}

We conclude the results in this subsection by the following theorem.

\begin{theorem}
We have
$$\mathfrak{B}(\mathbf{x}^{(0)})= \mathbb{Q}_2 \setminus \left(\mathbf{x}^{(1)} \cup \bigcup_{n=0}^{+\infty}f^{-n}_{\theta, q, 3}\left\{\mathbf{x}^{(\infty)}\right\}\right).$$
\end{theorem}

\begin{proof}
It follows from Propositions \ref{p=2} and \ref{p=2_1}.
\end{proof}
	
\bibliographystyle{amsplain}

\end{document}